\documentclass[onefignum,onetabnum]{siamart171218}

\usepackage{lipsum}
\usepackage{mathtools}
\usepackage{amsfonts}
\usepackage{graphicx}
\usepackage{epstopdf}
\usepackage{makecell}
\usepackage{anyfontsize}
\usepackage{t1enc}
\usepackage{algorithmic}
\ifpdf
  \DeclareGraphicsExtensions{.eps,.pdf,.png,.jpg}
\else
  \DeclareGraphicsExtensions{.eps}
\fi

\newcommand{\interior}[1]{%
  {\kern0pt#1}^{\mathrm{o}}%
}

\newcommand\paul[1]{\textcolor{red}{PG: #1}}

\newcommand{\bbR}{\mathbb{R}}

\newcommand{\bbE}{\mathbb{E}}

\newcommand\ignore[1]{}
\newcommand{\defeq}{\vcentcolon=}

\newsiamremark{remark}{Remark}
\newsiamremark{assumption}{Assumption}
\newsiamremark{hypothesis}{Hypothesis}
\crefname{hypothesis}{Hypothesis}{Hypotheses}
\newsiamthm{claim}{Claim}

\headers{Penalized Stochastic Gradient Methods}{M. Li, P. Grigas, A. Atamtürk}

\title{New Penalized Stochastic Gradient Methods for Linearly Constrained Strongly Convex Optimization}

\author{Meng Li\thanks{Department of Industrial Engineering and Operations Research, University of California-Berkeley, Berkeley, CA 
  (\email{meng\_li@berkeley.edu}, \email{pgrigas@berkeley.edu}, \email{atamturk@berkeley.edu}).}
\and Paul Grigas\footnotemark[1]
\and Alper Atamtürk\footnotemark[1]}

\usepackage{amsopn}

\makeatletter
\newcommand*{\addFileDependency}[1]{
  \typeout{(#1)}
  \@addtofilelist{#1}
  \IfFileExists{#1}{}{\typeout{No file #1.}}
}
\makeatother

\newcommand*{\myexternaldocument}[1]{%
    \externaldocument{#1}%
    \addFileDependency{#1.tex}%
    \addFileDependency{#1.aux}%
}

\newcommand{\calM}{\mathcal{M}}

\ifpdf
\hypersetup{
  pdftitle={New Penalized Stochastic Gradient Methods for Linearly Constrained Strongly Convex Optimization},
  pdfauthor={M. Li, P. Grigas, A. Atamtürk}
}
\fi

\myexternaldocument{ex_supplement}

\begin{document}

\maketitle

\begin{abstract}
For minimizing a strongly convex objective function subject to linear inequality constraints, we consider a penalty approach that allows one to utilize stochastic methods for problems with a large number of constraints and/or objective function terms. We provide upper bounds on the distance between the solutions to the original constrained problem and the penalty reformulations, guaranteeing the convergence of the proposed approach. We give a nested accelerated stochastic gradient method and propose a novel way for updating the smoothness parameter of the penalty function and the step-size. 
The proposed algorithm requires at most $\tilde O(1/\sqrt{\epsilon})$ expected stochastic gradient iterations to produce a solution within an expected distance of $\epsilon$ to the optimal solution of the original problem, which is the best complexity for this problem class to the best of our knowledge. We also show how to query an approximate dual solution after stochastically solving the penalty reformulations, leading to results on the convergence of the duality gap. Moreover, the nested structure of the algorithm and upper bounds on the distance to the optimal solutions allows one to safely eliminate constraints that are inactive at an optimal solution throughout the algorithm, which leads to improved complexity results. Finally, we present computational results that demonstrate the effectiveness and robustness of our algorithm. 
\end{abstract}



\section{Introduction}
Consider the convex optimization problem

\begin{equation}\label{Start}
\begin{aligned}
\min_{x\in\mathbb{R}^n} \quad & F(x)=\frac1{\ell}\sum_{i=1}^{\ell}f_i(x)+\psi(x)\\
\mathrm{s.t.} \quad & x \in X_i, \ i=1,\ldots,{m},\\
\end{aligned}
\end{equation}
where \(f_i : \mathbb{R}^n\to\mathbb{R}\), \(i=1,...,{\ell},\) are convex and smooth functions, 
\(\psi:\mathbb{R}^n\to\mathbb{R}\) is a convex function for which one can build a proximal mapping, and $X_i \subseteq \bbR^n$ is a nonempty closed convex set for $i = 1, \ldots, {m}$. Problems of the form \eqref{Start} arise in many contexts, including predictive control \cite{borrelli2017predictive}, portfolio optimization \cite{markowitz2000mean}, and others. In machine learning, some examples include isotonic regression \cite{barlow1972isotonic}, convex regression \cite{seijo2011nonparametric, lim2012consistency}, and strong convex relaxations of sparse combinatorial problems such as signal estimation \cite{atamturk2018sparse} and sparse regression \cite{AG:Mmatrix,atamturk2019rank,han20202x2}. We are particularly interested in problems where the number of objective terms $\ell$ and/or the number of constraints $m$ are very large.

In this paper, we consider the case where the objective function \(F\) is \(\mu\)-strongly convex for $\mu > 0$, and the feasible region is defined by a set of linear inequalities, i.e., \(X_i := \{x \in \bbR^n : a_i^Tx\le b_i\}\), \(i=1,...,{m}\). Let \(A=(a_1,...,a_{m})^T\in\mathbb{R}^{{m}\times n}\) and \(b=(b_1,...,b_{m})^T\in\mathbb{R}^{m}\). Then, our problem of interest is
\begin{equation}\label{eq:Original Linear}
\begin{aligned}
\min_{x\in\mathbb{R}^n}\quad &F(x)=\frac1{\ell}\sum_{i=1}^{\ell}f_i(x)+\psi(x)\\
\mathrm{s.t.}\quad &a_i^Tx\le b_i, \  i=1,...,{m}.
\end{aligned}
\end{equation}

Our objective herein is to devise a penalty-based stochastic first-order approach for solving problem \eqref{eq:Original Linear}. The main advantage of the penalty function approach is that the resulting penalized reformulation is an unconstrained, smooth, convex optimization problem whose objective involves a finite sum over the components of objectives and the penalty functions. Such problems with a finite sum structure are particularly amenable to stochastic (proximal) gradient methods (see, e.g., \cite{bottou2018optimization} and the references therein) and variants that are able to exploit the finite sum structure to achieve faster convergence such as the stochastic variance reduced gradient method (SVRG) \cite{johnson2013accelerating}, the stochastic average gradient methods (SAG, SAGA) \cite{schmidt2017minimizing,defazio2014saga}, and related methods.

In this work, we construct a penalty reformulation of 
\eqref{eq:Original Linear} based on using the \emph{softplus} penalty function, which arises from a smoothing \cite{nesterov2005smooth} of the $\max(0, \cdot)$ function. We construct novel estimates of the 2-norm distance between solutions of the penalty reformulation and the original constrained problem by analyzing the smooth solution trajectories and results on \(C^{\infty}\) approximations of convex (and strongly convex) functions due to Azagra \cite{azagra2013global}. Then, based on these estimation results, we analyze the complexity of solving the penalty reformulations with various methods, most notably the accelerated stochastic methods that exploit the finite-sum structure. Moreover, using the structure of the proposed nested penalty method and our novel bounds, we are able to obtain approximate dual solutions and we can safely eliminate inactive constraints throughout the course of the algorithm.

There are several existing approaches for solving problems of the form \eqref{Start}. Classical approaches include interior point and projected gradient methods, as well as the augmented Lagrangian \cite{bertsekas2014constrained} and ADMM methods \cite{boyd2011distributed}. These approaches do not scale well when ${\ell}$ and ${m}$ are very large, either because they require projecting onto or otherwise directly working with the intersection of all of the constraint sets or because they require calculating the gradient of a function that is the sum of many components. One approach for addressing the case where ${m}$ is very large is the gradient descent method with projections onto randomly sampled constraint sets $X_i$ \cite{nedic2011random,wang2015incremental}.

Penalty methods have been considered in several previous works on similar problems. Nedich and Tatarenko \cite{tatarenko2018smooth} 
consider problem \eqref{eq:Original Linear} (without finite-sum structure in the objective), use the one-sided Huber penalty to construct the penalty reformulations, and apply SAGA to solve the unconstrained penalty reformulation. An unfortunate deficiency of this approach is that the smoothness parameter of the penalty function as well as the weight parameter on the penalty function must satisfy a complicated set of relations in order to apply a linear convergence rate inherited by the SAGA algorithm. In a follow up paper \cite{nedic2020convergence}, they present an incremental gradient method that dynamically updates the parameters and achieves an asymptotic convergence rate of $O(1/\sqrt{k})$, where $k$ is the iteration counter and each iteration involves working with a single constraint. 
Instead, we use the softplus function as our smooth approximation, which leads to improved estimation results (\cref{thm:estimation linear new}). This improved estimation result, as well as our refined analysis, ultimately leads to an improved \(\tilde O(\frac1{\sqrt{\epsilon}})\) complexity bound in terms of the expected number of incremental steps (i.e., stochastic gradient and constraint evaluations) required to find a solution within a 2-norm distance of $\epsilon$ to the optimal solution.

Lan and Monteriro \cite{lan2013iteration}
consider convex problems with conic constraints \(Ax-b\in \mathcal{K}^*\), where \(\mathcal{K}^*\) is the dual cone of some closed convex cone \(\mathcal{K}\), and use the quadratic penalty \([d_{\mathcal{K}^*}(Ax-b)]^2=\min_{u\in\mathcal{K}^*}\|u+b-Ax\|_2^2\) to relax the constraints. Compared with their setting and methods, ($i$) we consider the case when \({\ell}\) and \({m}\) are large and focus on the complexity with respect to incremental steps (i.e., calls of \(f_i\), \(a_i\), and gradients thereof) and use stochastic methods; ($ii$) we use the softplus function, which is close to \(\max(a_i^Tx-b_i,0)\) instead of a quadratic function to penalize the constraints. Fercoq et al. \cite{fercoq2019almost} consider the objective to be an expectation of random smooth convex functions and the random constraints to be held almost surely, which can be thought as \eqref{eq:Original Linear} when \(m\) tends to infinity, and use a similar quadratic penalty as in \cite{lan2013iteration}. In addition to using a different penalty function, we apply the catalyst SVRG (or SAG, SAGA) method \cite{lin2015universal} which exploits the finite sum structure of the penalty reformulation to obtain a $\tilde O(1/\sqrt{\epsilon})$ expected complexity bound in terms of distance to the optimal solution. On the other hand, Fercoq et al. \cite{fercoq2019almost} demonstrate a $\tilde O(1/\epsilon)$ bound in terms of the objective function gap and 
the constraint violations (for the strongly convex case). Using the sum of squared distance functions \(\frac1{2m}\sum_{i=1}^md_{\mathcal{X}_i}(x)^2\) to penalize the constraints, Mishchenko and Richt\'{a}rik
\cite{mishchenko2018stochastic} obtain \(O(\frac1\epsilon)\) complexity, and their algorithm and results cover nonlinear constraints as well. Compared with their result, we have a better convergence rate in the linear case, and we do not require a global Hoffman-type assumption. 
 
\paragraph{Contributions} This paper has three main contributions. First, we propose to use the softplus function to penalize the constraints, and we show that this penalty function reformulation also arises by applying Nesterov's smoothing technique \cite{nesterov2005smooth} to the Lagrangian of the original problem. The penalty reformulation has a finite sum structure that allows one to employ stochastic gradient methods, and further, the relation with the Lagrangian enables us to query an approximately optimal dual solution after obtaining an approximately optimal primal solution. Second, we use ordinary differential equations techniques  to estimate the 2-norm distance between solutions to the penalty reformulations and the solution to \eqref{eq:Original Linear}, which leads to novel and stronger estimation bounds as compared to existing approaches.
Third, we analyze the complexity of accelerated stochastic methods to obtain approximate primal and dual solutions with our penalty method and compare it with the existing lower bounds on such problems. In addition, we propose to use a nested algorithm to solve the penalty reformulations, which ensures limiting convergence can help to determine the parameters during implementations. Moreover, based on the nested structure and upper bounds on the distance to the optimal solution, we design a screening procedure to effectively eliminate constrains that are inactive at the optimal solution throughout the algorithm, which leads to improved complexity results. 
We present the results of numerical experiments performed on quadratic programming problem and a SVM problems, to demonstrate the effectiveness of the proposed algorithm as compared to existing penalty function approaches.

\paragraph{Organization}
This paper is organized as follows. In \cref{sec:Penalty}, we introduce the penalty reformulation for the constrained problem \eqref{eq:Original Linear} and study its properties. Our nested algorithm based on the penalty reformulation and its complexity analysis is presented in \cref{sec:alg}. In \cref{sec:duality}, we show how to query dual solutions and present results on the convergence of the duality gap. In \cref{sec:screening}, we show how to incorporate the screening procedure to eliminate inactive constraints into our algorithm and improve the complexity. We present our numerical results to demonstrate the effectiveness of our proposed algorithm in \cref{sec:numerical}. We conclude with a few final remarks in \cref{sec:conclusions}.

\pagebreak

\section{Penalty Function Reformulation and Key Properties}
\label{sec:Penalty}

In this section, we examine several useful theoretical properties of a smooth penalized reformulation of problem \eqref{eq:Original Linear}. To penalize the constraints, we use the \textit{softplus} function, 
\begin{equation}\label{Softplus Penalty}
    p_{\delta}(t):=\delta\log(1+\exp(t/\delta)),\ (\delta,t)\in[0,+\infty)\times\mathbb{R},
\end{equation}
where $\delta \geq 0$ is a parameter controlling the smoothness of $p_{\delta}$.
Function $p_\delta$ is used as a loss function in many contexts, including, logistic regression. The softplus penalty may be viewed as an instance of Nesterov's smoothing technique~\cite{nesterov2005smooth} applied to the hinge function $\max(0,t)$ since it holds that
\begin{equation}\label{eq:softplus_nesterov}
    p_{\delta}(t) = \max_{r \in [0,1]}\big \{tr - \delta[r\log(r) + (1-r)\log(1-r)] \big\}.
\end{equation}
The above conjugate representation will be useful in studying the Lagrangian of \eqref{eq:Original Linear} and developing duality gap results in \cref{sec:duality}. 
\begin{figure}[ht]
\centering
\includegraphics[width=0.7\textwidth]{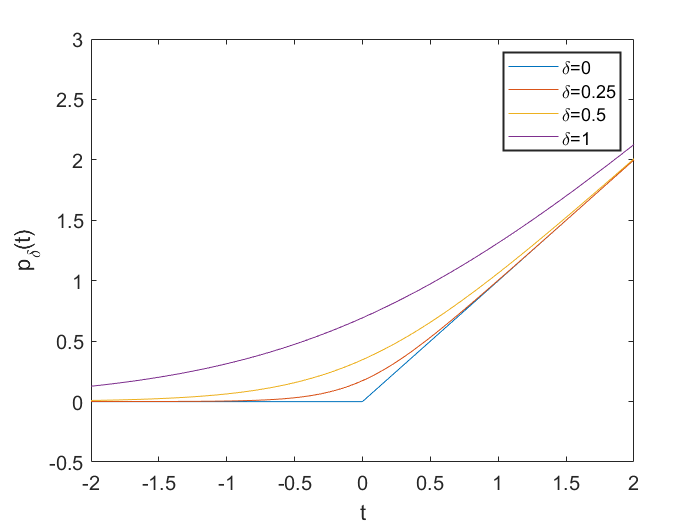}
\caption{Softplus function \(p_{\delta}\) for different values of $\delta$.}\label{Fig_softplus}
\end{figure}
\cref{Fig_softplus} shows the softplus function \(p_{\delta}\) for different values of $\delta$ and highlights how $\delta$ controls the trade-off between smoothness of \(p_{\delta}\) and closeness of the approximation to $p_0(t) = \max(0, t)$. \cref{prop:softplus} formalizes this intuition and provides additional properties of \(p_{\delta}\).

\begin{proposition}\label{prop:softplus}
The softplus function \(p_{\delta}\) satisfies the following properties:
\begin{enumerate}
\item $p_0(t) = \max(0, t)$ for all $t \in \bbR$;
\item $p_\delta$ is differentiable and convex for all $\delta > 0$;
\item $p^\prime_\delta(t) \in [0,1]$ for all $\delta > 0$ and $t \in \bbR$;
\item $|p^\prime_\delta(t_1) - p^\prime_\delta(t_2)| \leq \frac{1}{4\delta}|t_1 - t_2|$ for all $\delta > 0$ and $t_1, t_2 \in \bbR$;
\item $0 \! \leq \! p_{\delta_2}(t) \! - \! p_{\delta_1}(t) \! \leq \! p_{\delta_2}(0) \! - \! p_{\delta_1}(0) \! = \! (\delta_2 \! - \! \delta_1)\log 2$ for all $\delta_2 \geq \delta_1 \geq 0$ and $t \in \bbR$.
\end{enumerate}
\end{proposition}

The proof of \cref{prop:softplus} is included for completeness in \cref{sec:proofs}.

Let us now consider the penalty reformulation for \eqref{eq:Original Linear}  
\begin{equation}\label{eq:Penalty Linear New}
    \min_{x \in \bbR^n} F_{\xi ,\delta}(x):=F(x)+\xi  \sum_{i=1}^{m} p_{\delta}(a_i^Tx-b_i),
\end{equation}
where $\xi  \geq 0$ is the penalty parameter and $\delta \geq 0$ is the smoothness parameter of the softplus function. 
Let $x_{\xi ,\delta}^\ast$ denote the unique optimal solution of \eqref{eq:Penalty Linear New} for given $\xi  \geq 0$ and $\delta \geq 0$.  
We are interested in studying the relationship between the optimal solution $x^\ast$ for \eqref{eq:Original Linear} and $x_{\xi ,\delta}^\ast$ as we vary the parameters $\xi $ and $\delta$. Towards this goal, throughout the paper, the following assumptions are made concerning problem \eqref{eq:Original Linear}.

\begin{assumption}\label{assum:Assumption Linear New}
Problem \eqref{eq:Original Linear} satisfies the following properties:
\begin{enumerate}
\item The objective function $F$ is globally $\mu$-strongly convex for some $\mu > 0$, i.e., $F - \frac\mu2 \|\cdot\|_2^2$ is a convex function;
\item The component functions $f_i$, $i=1,...,{\ell}$, are convex and globally $L$-smooth for some $L \geq 0$, i.e.,  $\|\nabla f_i(x) - \nabla f_i(y)\|_2 \leq L\|x - y\|_2$ for all $x, y \in \bbR^n$;
\item The proximal term \(\psi\) is a proper convex and closed function;
\item The feasible region $\{x  \in \bbR^n: Ax \leq b\}$ is nonempty, and the rows of $A$ are normalized so that \(\|a_i\|_2 = 1\) for \(i=1,...,{m}\).
\end{enumerate}
\end{assumption}

Let us also consider the dual problem of \eqref{eq:Original Linear},
\begin{equation}\label{eq:Original Dual Linear}
\begin{aligned}
    \max_{\lambda \ge 0}\quad &G(\lambda) := -F^*(-A^T\lambda)-b^T\lambda, 
\end{aligned}
\end{equation}
where \(F^*(y) =\sup_{x}\left\{y^Tx-F(x)\right\}\) is the conjugate function of \(F\). Note that strong convexity of $F$ ensures that $F^\ast$ is a smooth function. Let $\Lambda := \{\lambda : \lambda \geq 0\}$ denote the feasible region of \eqref{eq:Original Dual Linear} and let $\Lambda^\ast$ denote the set of dual optimal solutions of \eqref{eq:Original Dual Linear}.
The following lemma shows that under \cref{assum:Assumption Linear New}, for a large enough \(\xi \ge0\), the solution of the penalty reformulation \eqref{eq:Penalty Linear New} at \(\delta=0\), \(x_{\xi ,0}^*\), equals the solution of the original problem \eqref{eq:Original Linear}, \(x^\ast\).
\begin{lemma}\label{lemma:0 point}
Define \(\bar{\xi} := \inf_{\lambda^\ast \in \Lambda^\ast}\|\lambda^\ast\|_\infty\). Then, $\bar{\xi}$ is finite and it holds that \(x_{\xi ,0}^*=x^*\) for all \(\xi \ge \bar{\xi}\).
\end{lemma}
\begin{proof}
    By the Lagrangian necessary conditions, there exists \(\lambda^*\in\bbR_{+}^{m}\) such that 
    \[0\in\partial F(x^*)+\sum_{i=1}^{m}\lambda_i^* a_i ,\]
    and $a_i^Tx^*-b_i = 0$ for all $i$ with $\lambda_i^* > 0$. Hence, $\Lambda^\ast$ is non-empty and $\bar{\xi}$ is finite.
    Given \(\bar{\xi} = \inf_{\lambda^\ast \in \Lambda^\ast}\|\lambda^\ast\|_\infty\), since \(\partial_t p_{0}(0)=[0,1]\), when \(\xi \ge \bar{\xi}\) it holds that
    \[0\in \partial F(x^*)+\xi \sum_{i=1}^{m}\partial_x p_0(a_i^Tx^*-b_i)=\partial_x F_{\xi ,0}(x^*).\]
    Thus, \(x^*\) is the unique solution of \eqref{eq:Penalty Linear New}.
\end{proof}
Tatarenko and Nedich \cite{tatarenko2018smooth} present a similar bound on their penalty parameter that is based on Hoffman's bound for the polyhedral feasible region, and Mishchenko and Richt\'{a}rik \cite{mishchenko2018stochastic} present a bound on their penalty parameter based on an assumption that is close to the definition of Hoffman's constant. In both cases, estimating the relevant Hoffman-type constant can be difficult in practice. On the other hand, the bound presented in \cref{lemma:0 point} is based on the infinity norm of a dual optimal solution, which may possibly be estimated in an adaptive manner.
One may also expect the lower bound in \cref{lemma:0 point} to be a less strict requirement because it is based on a local property of optimal solutions as compared to the Hoffman-type bounds that must hold globally.
Finally, in \cref{sec:numerical}, we observe that using conservatively large values of \(\xi \), e.g., values much larger than than the infinity norm of a dual optimal solution, does not have a significant effect on the performance of our algorithms.

Our next objective is to bound the 2-norm distance between \(x_{\xi ,\delta}^\ast\) and \(x_{\xi ,0}^\ast\). \cref{thm:estimation linear new} gives three such upper bounds on $\|x_{\xi ,\delta}^*-x_{\xi ,0}^*\|_2$. Throughout, let \(s_{\max}\) be the maximum singular value of the matrix \(A\), and let \(s_{\mathrm{pmin}}\) be the minimum positive singular value of \(A\).

\begin{theorem}\label{thm:estimation linear new}
The following upper bounds on $\|x_{\xi ,\delta}^*-x_{\xi ,0}^*\|_2$ hold. 
\begin{enumerate}
    \item For any \(\xi >0\) and any $\delta \geq 0$, it holds that
    \begin{equation*}
        \|x_{\xi ,\delta}^*-x_{\xi ,0}^*\|_2\le\sqrt{\frac{2{m}\xi \delta\log2}{\mu}} 
        \cdot
    \end{equation*}
    \item For any \(\xi >0\) and any \(\delta\in [0,\frac{\xi }{\mu}\exp(-2)]\), it holds that \[\|x_{\xi ,\delta}^*-x_{\xi ,0}^*\|_2\le {m}\delta\log\left(\frac{\xi }{\mu\delta}\right).\]
    \item For any \(\xi >0\) and  any \(\delta\in [0,\frac{s_{\max}\max({s_{\mathrm{pmin}}},1)\xi }{\mu}\exp(-1)]\), it holds that \[\|x_{\xi ,\delta}^*-x_{\xi ,0}^*\|_2\le \frac{4\sqrt{m}}{{s_{\mathrm{pmin}}}}\delta\log\left(\frac{s_{\max}\max({s_{\mathrm{pmin}}},1)\xi }{\mu\delta}\right).\]
\end{enumerate}
\end{theorem}

First, note that based on the penalty reformulation construction, and the \(5\)th property in \cref{prop:softplus}, 
\[F_{\xi ,\delta}(x)-F_{\xi ,0}(x)\in[0,{m}\xi \delta\log2],\quad \forall x \in \bbR^n.\] 
Then, the first claim follows from the \(\mu\)-strong convexity of $F$. Before presenting the proof of the remaining two claims, we introduce the following notation for residual vectors: 
\[\begin{aligned}
z_{\xi ,\delta}&=A(x_{\xi ,\delta}^*-x_{\xi, 0}^*),\\
\tilde z_{\xi ,\delta}&=Ax_{\xi ,\delta}^*-b.
\end{aligned}\]
Furthermore, the following simple result will be used often in the proofs.
\begin{lemma}\label{lemma:1d 1+exp}
Consider the scalar function $\phi_{c,d} : [0, \infty) \to [0, \infty)$ defined as 
\[
\phi_{c,d}(\theta):=\frac{\theta}{d+c\exp(\theta)} \text{ for \(\theta\in [0,\infty)\) for constants c,d>0.}
\]
Whenever \(\log(\frac{d}{c})\ge 2\), it holds \(\phi_{c,d}(\theta) \le(\log(\frac{d}{c})-1)/d\) for all $\theta\in [0,\infty)$.
\end{lemma}
The proof of \cref{lemma:1d 1+exp} is included for completeness in \cref{sec:proofs}.

\begin{proof}[Proof of Claim 2 in \cref{thm:estimation linear new}]
First, note it is sufficient to prove the proposition for the case when \(F\) is twice-continuous differentiable. Otherwise, for any \(\epsilon>0\), by Corollary 1.3 in \cite{azagra2013global}, one can construct \(\tilde F\) to be infinitely differentiable and an \(\epsilon\)-approximation of \(F\), i.e., \(F(x)-\epsilon\le \tilde F(x)\le F(x)\). Let \(\tilde x_{\xi ,\delta}^*\) be the solution to \(\min_x\tilde F_{\xi ,\delta}(x)=\tilde F(x)+\xi  \sum_{i=1}^{m} p_{\delta}(a_i^Tx-b_i)\). Then, for \(\delta>0\),
\[\tilde F_{\xi ,\delta}( x_{\xi ,\delta}^*) \le F_{\xi ,\delta}(x_{\xi ,\delta}^*) \le F_{\xi ,\delta}(\tilde x_{\xi ,\delta}^*)\le \tilde F_{\xi ,\delta}(\tilde x_{\xi ,\delta}^*) + \epsilon,\]
which means \(\|\tilde x_{\xi ,\delta}^* - x_{\xi ,\delta}^*\|_2\le\frac{2\epsilon}\mu\), for all \(\delta>0\).  Similarly, we can construct an \(\epsilon\)-approximation of $\tilde F_{\xi ,0}$, \(\hat F_{\xi ,0}\), and the corresponding solution \(\hat x_{\xi ,0}\). Then, \(\|\hat x_{\xi ,0}^*-\tilde x_{\xi ,0}^*\|^2_2\le\frac{2\epsilon}\mu\) and \(\|\hat x_{\xi ,0}^*-x_{\xi ,0}^*\|^2_2\le\frac{4\epsilon}\mu\). Hence, \[\|x_{\xi ,\delta}^*-x_{\xi ,0}^*\|_2\le \|\tilde x_{\xi ,\delta}^*-\tilde x_{\xi ,0}^*\|_2 +(2+2\sqrt 2)\sqrt{\frac{\epsilon}{\mu}} \cdot\]
Because the choice of \(\epsilon\) is arbitrary, we may assume $F$ to be twice-continuous differentiable without loss of generality.

For simplicity of notation, in the proof, we let \(\xi >0\) to be a constant and drop it from \(x_{\xi ,\delta}^*\), \(z_{\xi ,\delta}\), and \(\tilde z_{\xi ,\delta}\).
Now, consider \(x_{\delta}^*\) as a function of \(\delta\). We have \(\nabla F_{\xi ,\delta}(x_{\delta}^*)=0\). Then,
\[\nabla^2 F_{\xi ,\delta}(x_{\delta}^*)\frac{d x_{\delta}^*}{d\delta}+\frac{\partial}{\partial\delta}\nabla F_{\xi ,\delta}(x_{\delta}^*)=0.\]
And since
\[\begin{aligned}
\nabla F_{\xi ,\delta}(x)&=\nabla F(x)+\xi \sum_{i=1}^m\nabla p_{\delta}(a_i^Tx-b_i)\\
&=\nabla F(x)+\xi \sum_{i=1}^m a_i \frac{\exp((a_i^Tx-b_i)/\delta)}{1+\exp((a_i^Tx-b_i)/\delta)} \cdot
\end{aligned}\]
Then,
\[\frac{\partial}{\partial\delta}\nabla F_{\xi ,\delta}(x)=\xi \sum_{i=1}^m a_i\frac{-\frac{a_i^Tx-b_i}{\delta^2}\exp((a_i^Tx-b_i)/\delta)}{(1+\exp((a_i^Tx-b_i)/\delta))^2},\]
\begin{equation}\label{eq:Hessian, linear, infty, m}
    \nabla^2 F_{\xi ,\delta}(x)=\nabla^2 F(x)+\xi \sum_{i=1}^m \frac{\frac{a_ia_i^T}{\delta}\exp((a_i^Tx-b_i)/\delta)}{(1+\exp((a_i^Tx-b_i)/\delta))^2},
\end{equation}
we have
\begin{equation}\label{eq:Multiple Constraints Derivative}
\begin{aligned}
\frac{d x_{\delta}^*}{d\delta}=&\left(\delta \nabla^2 F(x^*_{\delta})+\xi \sum_{i=1}^m \frac{a_ia_i^T\exp((a_i^Tx_{\delta}^*-b_i)/\delta)}{(1+\exp((a_i^Tx_{\delta}^*-b_i)/\delta))^2}\right)^{-1}\\&\left(\xi \sum_{i=1}^m a_i\frac{\frac{a_i^Tx_{\delta}^*-b_i}{\delta}\exp((a_i^Tx_{\delta}^*-b_i)/\delta)}{(1+\exp((a_i^Tx_{\delta}^*-b_i)/\delta))^2}\right).
\end{aligned}
\end{equation}
In other words,
\[\begin{aligned}
&\left\|\frac{d x_{\delta}^*}{d\delta}\right\|_2 \le \left\|\left(\mu\delta I+\xi \sum_{i=1}^m \frac{a_ia_i^T\exp(\tilde z_{\delta,i}/\delta)}{(1+\exp(\tilde z_{\delta,i}/\delta))^2}\right)^{-1} \xi \sum_{i=1}^m a_i\frac{\frac{\tilde z_{\delta,i}}{\delta}\exp(\tilde z_{\delta,i}/\delta)}{(1+\exp(\tilde z_{\delta,i}/\delta))^2}\right\|_2\\
&\le \sum_{i=1}^m\xi \left\|\left(\mu\delta I+\xi  \frac{a_ia_i^T\exp(\tilde z_{\delta,i}/\delta)}{(1+\exp(\tilde z_{\delta,i}/\delta))^2}\right)^{-1}a_i\right\|_2\frac{\frac{\tilde z_{\delta,i}}{\delta}\exp(|\tilde z_{\delta,i}|/\delta)}{(1+\exp(\tilde z_{\delta,i}/\delta))^2},
\end{aligned}\]
where by Woodbury's identity, and the fact \(\|a_i\|_2=1\)
\[\begin{aligned}
&\left\|\left(\mu\delta I+\xi  \frac{a_ia_i^T\exp(\tilde z_{\delta,i}/\delta)}{(1+\exp(\tilde z_{\delta,i}/\delta))^2}\right)^{-1}a_i\right\|_2\\
=& \left\|\left(\frac{1}{\mu\delta}I- \frac{1}{\mu^2\delta^2}a_i\left(\frac{a_i^Ta_i}{\mu\delta}+\frac{(1+\exp(\tilde z_{\delta,i}/\delta))^2}{\xi \exp(\tilde z_{\delta,i}/\delta)}\right)^{-1}a_i^T\right)a_i\right\|_2\\
=&  \frac{1}{\mu\delta}\frac{(1+\exp(\tilde z_{\delta,i}/\delta))^2}{\xi \exp(\tilde z_{\delta,i}/\delta)}\left(\frac{1}{\mu\delta}+\frac{(1+\exp(\tilde z_{\delta,i}/\delta))^2}{\xi \exp(\tilde z_{\delta,i}/\delta)}\right)^{-1}.
\end{aligned}\]
Hence, it holds
\[\begin{aligned}
\left\|\frac{d x_{\delta}^*}{d\delta}\right\|_2 &\le \sum_{i=1}^m \frac{|\tilde z_{\delta,i}|/\delta}{1+\frac{\mu\delta}{\xi }\frac{(1+\exp(\tilde z_{\delta,i}/\delta))^2}{\exp(\tilde z_{\delta,i}/\delta)}}\le \frac{|\tilde z_{\delta,i}|/\delta}{1+\frac{\mu\delta}{\xi }\exp(|\tilde z_{\delta,i}|/\delta)} \cdot
\end{aligned}\]
Then, by \cref{lemma:1d 1+exp},
\[\left\|\frac{d x_{\delta}^*}{d\delta}\right\|_2\le m\left(\log\left(\frac{\xi }{\mu\delta}\right)-1\right),\quad \forall\delta\le \frac{\xi }{\mu}\exp(-2).\]
Hence, whenever \(\delta\in[0,\frac{\xi }{\mu}\exp(-2)]\),
it holds 
\(\|x_{\delta}^*-x_0^*\|_2\le m\delta\log\left(\frac{\xi }{\mu\delta}\right)\).

\end{proof}

Before presenting the proof for Claim 3 in \cref{thm:estimation linear new}, we first show the following proposition on the estimation of \(\|z_{\xi ,\delta}\|_2\), which then leads to an upper bounds on the violation of constraints at \(x_{\xi ,\delta}^*\).   

\begin{proposition}\label{prop:estimation linear new z infty m}
For any \(\xi >0\) and \(\delta\in [0,\frac{s_{\max}^2\xi }{\mu}\exp(-2)]\), it holds that 
\[\|z_{\xi ,\delta}\|_2\le \sqrt{m}\delta\log\left(\frac{s_{\max}^2\xi }{\mu\delta}\right) \cdot \]
\end{proposition}
\begin{proof}
As before, it suffices to prove the result for the case when \(F(x)\) is twice-continuously differentiable. For simplicity of notation,  we let \(\xi >0\) to be a constant and use the simpler notations \(x_{\delta}^*\), \(z_{\delta}\), \(\tilde z_{\delta}\) in the proof. Letting 
\[\Pi_{\tilde z_{\delta},\delta}=\mathrm{diag} \bigg (\frac{\exp(\tilde z_{\delta,1}/\delta)}{(1+\exp(\tilde z_{\delta,1}/\delta))^2},...,\frac{\exp(\tilde z_{\delta,m}/\delta)}{(1+\exp(\tilde z_{\delta,m}/\delta))^2} \bigg),\]
we can express \eqref{eq:Multiple Constraints Derivative} as
\begin{equation}\label{eq:Multiple Constraints Derivative Matrix Version}
\frac{dx_{\delta}^*}{d \delta}=\xi (\delta\nabla^2F(x_{\delta}^*)+\xi A^T\Pi_{\tilde z_{\delta},\delta}A)^{-1}A^T\Pi_{\tilde z_{\delta},\delta}\frac{\tilde z_{\delta}}{\delta} \cdot    
\end{equation}
Then, 
\[
\frac{dz_{\delta}}{d \delta}=\xi A(\delta\nabla^2F(x_{\delta}^*)+\xi A^T\Pi_{\tilde z_{\delta},\delta}A)^{-1}A^T\Pi_{\tilde z_{\delta},\delta}\frac{\tilde z_{\delta}}{\delta} \cdot
\]
By Woodbury's identity,
\[\begin{aligned}
&(\delta\nabla^2F(x_{\delta}^*)+\xi A^T\Pi_{\tilde z_{\delta},\delta}A)^{-1}=(\delta\nabla^2F(x_{\delta}^*))^{-1}\\&-(\delta\nabla^2F(x_{\delta}^*))^{-1}A^T(A(\delta\nabla^2F(x_{\delta}^*))^{-1}A^T+\xi ^{-1}\Pi_{\tilde z_{\delta},\delta}^{-1})^{-1}A(\delta\nabla^2F(x_{\delta}^*))^{-1};
\end{aligned}\]
so
\[\begin{aligned}
\frac{dz_{\delta}}{d \delta}
=&A(\delta\nabla^2F(x_{\delta}^*))^{-1}A^T(A(\delta\nabla^2F(x_{\delta}^*))^{-1}A^T+\xi ^{-1}\Pi_{\tilde z_{\delta},\delta}^{-1})^{-1}\frac{\tilde z_{\delta}}{\delta}\\
=&A\left(\frac{\nabla^2 F(x_{\delta}^*)}{\mu}\right)^{-1}A^T\left(A(\frac{\nabla^2 F(x_{\delta}^*)}{\mu})^{-1}A^T+\frac{\mu\delta}{\xi }\Pi_{\tilde z_{\delta},\delta}^{-1}\right)^{-1}\tilde z_{\delta}/\delta.
\end{aligned}\]
Then, because \(A(\frac{\nabla^2 F(x_{\delta}^*)}{\mu})^{-1}A^T\preceq s_{\max}^2I\), we have
\[\left\|\frac{dz_{\delta}}{d \delta}\right\|_2\le\left\|\left(\frac{|\tilde z_{\delta,1}|/\delta}{1+\frac{\mu\delta}{s_{\max}^2\xi }\exp(|\tilde z_{\delta,1}|/\delta)},...,\frac{|\tilde z_{\delta,m}|/\delta}{1+\frac{\mu\delta}{s_{\max}^2\xi }\exp(|\tilde z_{\delta,m}|/\delta)}\right)\right\|_2.\]
By \cref{lemma:1d 1+exp}, for \(\delta\in [0,\frac{s_{\max}^2\xi }{\mu}\exp(-2)]\), it holds
\[\frac{|\tilde z_{\delta,i}|/\delta}{1+\frac{\mu\delta}{s_{\max}^2\xi }\exp\left(|\tilde z_{\delta,i}|/\delta\right)}\le \log\left(\frac{s_{\max}^2\xi }{\mu\delta}\right)-1.\]
Hence, \(\left\|\frac{dz_{\delta}}{d \delta}\right\|_2\le \sqrt{m}\left(\log\left(\frac{s_{\max}^2\xi }{\mu\delta}\right)-1\right)\), and  \(\|z_{\delta}\|_2\le\sqrt{m}\delta\log(\frac{s_{\max}^2\xi }{\mu\delta})\) follows.
\end{proof}
When \(\xi \ge \bar{\xi}\), by \cref{lemma:0 point}, \(x_{\xi ,\delta}^*=x^*\) and \(\tilde z_{\xi ,0}=Ax^*-b\le 0\). Then, by definition, \(\|(\tilde z_{\xi ,\delta})^+\|_2\le \|z_{\xi ,\delta}\|_2\), and \cref{prop:estimation linear new z infty m} shows that the violation of constrains for the optimal solution to the penalty reformulation \(\|(Ax_{\xi ,\delta}^*-b)^+\|\), is bounded by \(\tilde O(\sqrt{m}\delta)\).

\begin{proof}[Proof of Claim 3 in \cref{thm:estimation linear new}]
Similarly, we may assume that \(F(x)\) is twice-continuously differentiable. For simplicity of notation, we again let \(\xi >0\) to be a constant and use the simplified notations \(x_{\delta}^*\), \(z_{\delta}\), \(\tilde z_{\delta}\) in the proof. Based on the proof of \cref{prop:estimation linear new z infty m}, it holds
\[\begin{aligned}
\left\|\frac{dx_{\delta}}{d \delta}\right\|_2
\le\left\|A^T\left(AA^T+\frac{\mu\delta}{\xi }\Pi_{\tilde z_{\delta},\delta}^{-1}\right)^{-1}\frac{\tilde z_{\delta}}{\delta}\right\|_2.
\end{aligned}\]
Let \(\hat z_{\delta}=(\hat z_{\delta,1},...,\hat z_{\delta,1})^T\), where 
\[\hat z_{\delta,i}=\left\{\begin{aligned}
&\tilde z_{\delta_i}\quad &\text{if } |\tilde z_{\delta_i}|\ge \delta\log\left(\frac{s_{\max}\max({s_{\mathrm{pmin}}},1)\xi }{\mu\delta}\right)\\
&0\quad &\text{o.w.}
\end{aligned}\right. \]
Then, 
\[\begin{aligned}
\left\|\frac{dx_{\delta}}{d \delta}\right\|_2
&\le\left\|A^T(AA^T+\frac{\mu\delta}{\xi }\Pi_{\tilde z_{\delta},\delta}^{-1})^{-1}\frac{\hat z_{\delta}}{\delta}\right\|_2 + \left\|A^T\left(AA^T+\frac{\mu\delta}{\xi }\Pi_{\tilde z_{\delta},\delta}^{-1}\right)^{-1}\frac{\tilde z_{\delta}-\hat z_{\delta}}{\delta}\right\|_2\\
&\le \left\|A^T\Pi_{\tilde z_{\delta},\delta}\frac{\hat z_{\delta}}{\delta}\right\|_2 + \left\|A^T\left(AA^T+\frac{\mu\delta}{\xi }I\right)^{-1}\frac{\tilde z_{\delta}-\hat z_{\delta}}{\delta}\right\|_2\\
&\le s_{\max}\left\|\frac{\xi }{\mu\delta}\Pi_{\tilde z_{\delta},\delta}\frac{\hat z_{\delta}}{\delta}\right\|_2+ \left\|\frac{\tilde z_{\delta}-\hat z_{\delta}}{\delta}\right\|_2/{s_{\mathrm{pmin}}}.
\end{aligned}\]
By the definition of \(\hat z_{\delta}\), for \(\delta\in[0,\frac{s_{\max}\max({s_{\mathrm{pmin}}},1)\xi }{\mu}\exp(-1)]\),
\[\begin{aligned}
\left\|\frac{\xi }{\mu\delta}\Pi_{\tilde z_{\delta},\delta}\frac{\hat z_{\delta}}{\delta}\right\|_2 &\le \frac{\sqrt{m}}{{s_{\mathrm{pmin}}}s_{\max}},\\
\left\|\frac{\tilde z_{\delta}-\hat z_{\delta}}{\delta}\right\|_2 &\le \sqrt{m}\log\left(\frac{s_{\max}\max({s_{\mathrm{pmin}}},1)\xi }{\mu\delta}\right).
\end{aligned}
\]
Hence, the claim holds.
\end{proof}

\section{Nested Penalty Method and Complexity Analysis}
\label{sec:alg}

In this section, we present a generic algorithmic framework for solving the original constrained optimization problem \eqref{eq:Original Linear} via algorithms that solve the unconstrained penalty reformulation \eqref{eq:Penalty Linear New}. Our primary algorithm is a dynamic nested method (\cref{algo_nested_linear}), which solves a sequence of penalty reformulation problems as the smoothing parameter $\delta_k$ tends to zero at an appropriate rate.

We first present an algorithmic framework that uses a generic unconstrained convex optimization method $\calM$ to solve the penalty reformulation problems. We present a corresponding complexity result expressed in terms of the complexity of $\calM$ and our complexity analysis utilizes the results established in \cref{sec:Penalty} concerning bounds between the solutions of the original constrained problem and the penalty formulation. We are primarily interested in the use of fast stochastic gradient methods to solve the penalty reformulations; therefore, we then specialize our analysis to the exemplary case where $\calM$ is the proximal SVRG method with catalyst acceleration \cite{lin2015universal}.
In this case, we show that our algorithm requires at most \(\tilde O(1/\sqrt{\epsilon})\) total stochastic gradient evaluations to achieve a solution within a 2-norm distance of $\epsilon$ to the optimal solution $x^\ast$ to \eqref{eq:Original Linear}. While the complexity analysis is focused primarily on the use the proximal SVRG method with catalyst acceleration \cite{lin2015universal} as a black box to solve the penalty reformulations \eqref{eq:Penalty Linear New}, it is not strongly tied to the choice of this method; other accelerated methods for finite sum problems, such as Katyusha \cite{allen2017katyusha} and the RPDG method \cite{lan2018optimal}, may be utilized as well.

Recall that the objective function of the penalty reformulation \eqref{eq:Penalty Linear New} decomposes into the sum of three terms: {\em (i)} the average of $\ell$ smooth components of the objective function, {\em (ii)} the possibly non-smooth proximal term $\psi$, and {\em (iii)} the weighted sum of smooth penalty functions $p_\delta$ applied to the constraints. Let \(\tilde{F}_{\xi,\delta}(x) := \frac{1}{\ell}\sum_{i=1}^{\ell}f_i(x)+\xi  \sum_{i=1}^{m} p_{\delta}(a_i^Tx-b_i)\) denote the smooth terms so that the objective function $F_{\xi,\delta}$ of \eqref{eq:Penalty Linear New} decomposes as \(F_{\xi,\delta}(x) = \tilde{F}_{\xi,\delta}(x)+\psi(x)\).  
In our generic algorithm, we will utilize the proximal operator and generalized gradient mapping induced by the proximal objective component \(\psi\). Recall that the scaled proximal operator of $\psi$ with parameter $\alpha > 0$ is defined by
\[\mathrm{prox}_{\psi}(\bar{x}; \alpha) := \arg\min_{x \in \bbR^n}\left\{\psi(x)+\frac1{2\alpha}\|x-\bar{x}\|_2^2\right\} \quad \text{for all } \bar{x} \in \bbR^n.\]
The above scaled proximal operator $\mathrm{prox}_{\psi}(\cdot; \alpha)$ is equivalent to the standard proximal operator of the function $\alpha\psi$. The scaled proximal operator $\mathrm{prox}_{\psi}(\cdot; \alpha)$ induces a (generalized) gradient mapping $g_\psi$ that can be evaluated on any smooth convex function $\tilde{F} : \bbR^n \to \bbR$.
For any $\bar{x} \in \bbR^n$ and $\alpha > 0$, the gradient mapping induced by $\psi$ is defined by $g_{\psi}(\bar{x}; \tilde{F}, \alpha) := \frac{1}{\alpha}(\bar{x} - \tilde{x})$, where
\begin{align*}
\tilde{x} :=&~ \mathrm{prox}_{\psi}(\bar{x} - \alpha\nabla\tilde{F}(\bar{x}); \alpha) \\
=&~ \arg\min_{x \in \bbR^n}\left\{\tilde{F}(\bar{x}) + \nabla\tilde{F}(\bar{x})^T(x - \bar{x}) + \frac{1}{2\alpha}\|x - \bar{x}\|_2^2 + \psi(x) \right\} \cdot
\end{align*}
Note that the above notion of gradient mapping based on the proximal operator of the function $\psi$ is considered by Lin et al. \cite{lin2018catalyst} and is a generalization of the gradient mapping introduced by Nesterov \cite{nesterov2018lectures}. The intuition of the gradient mapping is that, when the function $\tilde{F}$ is smooth and strongly convex, the squared norm of the gradient mapping is directly related to the suboptimality gap, as formalized by the following Lemma.

\ignore{
\begin{lemma}\label{lemma:stopping criteria:general proximal old}
Let \(\psi:\mathbb{R}^n\to\mathbb{R}\cup\{\infty\}\) be a proper, convex, and closed proximal function with gradient mapping $g_\psi$, and let $\tilde{F} : \bbR^n \to \bbR$ be convex and $L$-smooth for some $L\ge 0$. Consider $F := \tilde{F} + \psi$, and let $F^\ast := \min_{x \in \bbR^n}F(x)$. Suppose \(F\) is \(\mu\)-strongly convex for some \(0\le \mu\le L \). Then, for any $\alpha \in (0, 1/L]$ and for any $\bar{x} \in \bbR^n$ with $\tilde{x} := \mathrm{prox}_{\psi}(\bar{x} - \alpha\nabla\tilde{F}(\bar{x}); \alpha)$, the following inequalities hold: \\
\begin{enumerate}
    \item \(F(x)\ge F(\tilde x) + g_{\psi}(\bar{x}; \tilde{F}, \alpha)^T(x - \bar{x}) + \frac{\alpha}{2}\|g_{\psi}(\bar{x}; \tilde{F}, \alpha)\|_2^2 + \frac{\mu}{2}\|x - \bar{x}\|_2^2,\) \( \forall x\in\mathbb{R}^n\);
    \item $\|g_{\psi}(\bar{x}; \tilde{F}, \alpha)\|_2^2 \leq \frac{2}{\alpha}(F(\bar{x}) - F^\ast)$;
    \item $\|g_{\psi}(\bar{x}; \tilde{F}, \alpha)\|_2^2 \geq 2\mu(F(\tilde{x}) - F^\ast)$.
\end{enumerate}
\end{lemma}}

\begin{lemma}\label{lemma:stopping criteria:general proximal}
Let \(\psi:\mathbb{R}^n\to\mathbb{R}\cup\{\infty\}\) be a proper, convex, and closed proximal function with gradient mapping $g_\psi$, and let $\tilde{F} : \bbR^n \to \bbR$ be convex and $L$-smooth for some $L\ge 0$. Consider $F := \tilde{F} + \psi$, and let $F^\ast := \min_{x \in \bbR^n}F(x)$. Suppose \(F\) is \(\mu\)-strongly convex for some \(0\le \mu\le L \). Then, for any $\alpha \in (0, 1/2L]$ and for any $\bar{x} \in \bbR^n$ with $\tilde{x} := \mathrm{prox}_{\psi}(\bar{x} - \alpha\nabla\tilde{F}(\bar{x}); \alpha)$, the following inequalities hold: \\
\begin{enumerate}
    \item $\|g_{\psi}(\bar{x}; \tilde{F}, \alpha)\|_2^2 \leq \frac{2}{\alpha}(F(\bar{x}) - F^\ast)$;
    \item $\|g_{\psi}(\bar{x}; \tilde{F}, \alpha)\|_2^2 \geq 2\mu(F(\tilde{x}) - F^\ast)$.
\end{enumerate}
\end{lemma}

\cref{lemma:stopping criteria:general proximal} is an extension of Theorem 2.2.13 of Nesterov \cite{nesterov2018lectures}, which covers the case when \(\psi(x)\) is an indicator function of a convex set. Part (2.) of \cref{lemma:stopping criteria:general proximal} is the same as Lemma 2 of \cite{lin2018catalyst}.
The proof of \cref{lemma:stopping criteria:general proximal} is included, for completeness, in \cref{sec:proofs}.

We presume herein that the unconstrained convex optimization method $\calM$, which is applied to the penalty reformulations \eqref{eq:Penalty Linear New} in our generic algorithm, is a randomized iterative method, such as a stochastic first-order method. That is, whenever we apply $\calM$ to minimize a convex function $F : \bbR^n \to \bbR$, the iterate sequence $\{x_k\}$ generated by $\calM$ is a stochastic process. For any convex function $F : \bbR^n \to \bbR$ such that $F^\ast := \inf_{x \in \bbR^n} F(x) > -\infty$ and for any $\epsilon > 0$, define the stopping time
\begin{equation}\label{eq:stopping_time_f}
T_{\calM, F}(\epsilon; x_0) := \inf\{k \geq 0 : F(x_k) - F^\ast \leq \epsilon\},
\end{equation}
where $\{x_k\}$ is the sequence generated by $\calM$ initialized at $x_0$ and applied to minimize $F$. Notice that $T_{\calM, F}(\epsilon;x_0)$ is a random variable when $\calM$ is a randomized method. We assume throughout that $\bbE[T_{\calM,F}(\epsilon;x_0)] < \infty$ for any $\epsilon > 0$. In the case when the convex function $F$ decomposes as $F = \tilde{F} + \psi$ as in the conditions of \cref{lemma:stopping criteria:general proximal}, we can also naturally define, for any fixed $\alpha > 0$ and $\epsilon^\prime > 0$, a stopping time in terms of the gradient mapping:
\begin{equation}\label{eq:stopping_time_map}
\tilde{T}_{\calM, F}(\epsilon^\prime ; x_0, \alpha) := \inf \big \{k \geq 0 : \|g_{\psi}(x_k; \tilde{F}, \alpha)\|_2 \leq \epsilon^\prime \big \}.
\end{equation}
Items (2.)\ and (3.)\ of \cref{lemma:stopping criteria:general proximal} clearly imply that these two distinct notions of stopping times can be related to each other. The main advantage of the stopping time \eqref{eq:stopping_time_map} over \eqref{eq:stopping_time_f} is that \eqref{eq:stopping_time_map} provides an implementable stopping criterion even when $F^\ast$ is unknown, as long as the gradient mapping is computable. The previously defined stopping times will be useful in the analysis of our nested algorithm, which involves several stages where $\calM$ is applied separately at each stage.

We are now ready to present our generic nested penalty method in \cref{algo_nested_linear}. \cref{algo_nested_linear} requires specifying a value for the penalty parameter $\xi > 0$ and a value for the ``smoothness scaling paramter'' $\eta > 1$ which controls the rate at which we shrink the smoothness parameter $\delta$ towards zero. The unconstrained convex optimization method $\calM$ as well as its total number of iterations at each stage $\tau_t$ are both currently left as generic. Note that, after finding an approximate solution $\hat{x}_t$ using method $\calM$, an additional computation is performed to compute the proximal operator as required in the calculation of the gradient mapping at $\hat{x}_t$.

\begin{algorithm}
\caption{Generic Nested Penalty Method}
\label{algo_nested_linear}
\begin{algorithmic}
    \STATE {\bf Parameters:}  Penalty parameter $\xi > 0$, smoothness scaling parameter \(\eta>1\), and generic unconstrained convex optimization method $\calM$.
    \STATE {\bf Initialize:} $\hat{x}_{-1} \in \bbR^n$, $\delta_{0} > 0$.
    \FOR{$t = 0, 1, \ldots, T$}
    \STATE  1. Apply method $\calM$ on the penalty reformulation problem \eqref{eq:Penalty Linear New}, with penalty parameter $\xi$ and smoothness parameter $\delta_t$, initialized at $\tilde{x}_{t-1}$ and for $\tau_t \geq 1$ total iterations, to obtain an approximate solution
    \begin{equation*}
        \hat{x}_t \approx \arg\min_{x \in \bbR^n} \left\{F_{\xi ,\delta_t}(x) := F(x) + \xi\sum_{i=1}^{m} p_{\delta_t}(a_i^Tx-b_i)\right\} \cdot
    \end{equation*}
    \STATE 2. Choose step-size $\alpha_t > 0$ and set $\tilde{x}_t \gets \mathrm{prox}_{\psi}(\hat{x}_t - \alpha_t\nabla\tilde{F}_{\xi,\delta_t}(\hat{x}_t); \alpha_t)$.
    \STATE 3. Update $\delta_{t+1} \gets \delta_{t}/\eta$.
    \ENDFOR
\end{algorithmic}
\end{algorithm}

Our main result concerning the generic \cref{algo_nested_linear} is presented in \cref{lem:Nested Complexity-Generic}, which shows that the sequence $\{\tilde{x}_t\}$ converges linearly to the optimal solution $x^\ast$ of \eqref{eq:Original Linear} when the subproblems are solved to a sufficiently small accuracy.

\begin{lemma}\label{lem:Nested Complexity-Generic}
Consider applying the generic nested penalty method (\cref{algo_nested_linear}) using penalty parameter \(\xi \ge \bar{\xi}\) and initial smoothness parameter $\delta_0 \leq \frac{\xi}{\mu}\exp(-2)$. Let $\{\epsilon_t\}$ be an accuracy sequence for the penalty reformulation subproblems satisfying
\begin{equation*}
0 < \epsilon_t \leq \frac12 m\mu\delta_t^2\log^2\left(\frac{\xi s_{\max}}{\mu\delta_t}\right) \text{ for } t \geq 0,
\end{equation*}
and suppose that stopping criterion  \eqref{eq:stopping_time_map} is used with step-size \(0 < \alpha_t \leq 1/(L+\frac{\xi s_{\max}^2}{4\delta_t})\), i.e., $\tau_t = \tilde{T}_{\calM, F_{\xi ,\delta_t}}(\sqrt{2\mu\epsilon_t}; \tilde{x}_{t-1}, \alpha_t)$ for all $t \geq 0$.
Then, for any $t \geq 0$, it holds that:
\begin{equation}\label{eq:generic_convergence}
\begin{aligned}
\|\tilde{x}_t - x^\ast\|_2 &\leq \left(\min\left\{m,\frac{4\sqrt{m}}{{s_{\mathrm{pmin}}}}\right\}+\sqrt{m}\right) \delta_t \log\left(\frac{\xi s_{\max}}{\mu\delta_t}\right)\\&=\frac{\left(\min\left\{m,\frac{4\sqrt{m}}{{s_{\mathrm{pmin}}}}\right\}+\sqrt{m}\right)\delta_0}{\eta^t}\left[t\log(\eta) + \log\left(\frac{\xi s_{\max}}{\mu\delta_0}\right)\right].
\end{aligned}
\end{equation}
If instead the number of inner iterations $\tau_t$ is chosen so that $F_{\xi ,\delta_t}(\hat{x}_t) - F_{\xi ,\delta_t}^\ast \leq \epsilon_t$ for all $t \geq 0$, then \eqref{eq:generic_convergence} also holds.
\end{lemma}

\begin{proof}
Note that, when $\tau_t = \tilde{T}_{\calM, F_{\xi ,\delta_t}}(\sqrt{2\mu\epsilon_t}; \tilde{x}_{t-1}, \alpha_t)$, item (2.)\ of \cref{lemma:stopping criteria:general proximal} implies $F_{\xi ,\delta_t}(\tilde{x}_t) - F_{\xi ,\delta_t}^\ast \leq \epsilon_t$. Furthermore, when it is assumed that $F_{\xi ,\delta_t}(\hat{x}_t) - F_{\xi ,\delta_t}^\ast \leq \epsilon_t$, then item (1.) of \cref{lemma:stopping criteria:general proximal} implies $F_{\xi ,\delta_t}(\tilde{x}_t) - F_{\xi ,\delta_t}^\ast \leq F_{\xi ,\delta_t}(\hat{x}_t) - F_{\xi ,\delta_t}^\ast \leq \epsilon_t$. Thus, under both cases we have that $F_{\xi ,\delta_t}(\tilde{x}_t) - F_{\xi ,\delta_t}^\ast \leq \epsilon_t$. Furthermore, it holds that 
\begin{equation*}
\|\tilde{x}_t - x^\ast\|_2 \leq \|\tilde{x}_t - x_{\xi ,\delta_t}^*\|_2 + \|x_{\xi ,\delta_t}^* - x^*\|_2 \leq \sqrt{\frac{2\epsilon_t}{\mu}} + \|x_{\xi ,\delta_t}^* - x^*\|_2.
\end{equation*}
Then, the result holds by \(\mu\)-strong convexity of \(F_{\xi,\delta_t}\), and parts 2--3 of \cref{thm:estimation linear new}. 
\end{proof}

Note that \cref{lem:Nested Complexity-Generic} requires availability of the constants $\bar{\xi}$, $\mu$, and $s_{\max}$ and in order to set the penalty parameter $\xi$ and the sequences of accuracy parameters $\epsilon_t$ and step-sizes $\alpha_t$. As mentioned previously, using a conservatively large value of $\xi \geq \bar{\xi}$ is not prohibitive in practice, as observed in the numerical experiments presented in \cref{sec:numerical}. Likewise, depending on the information available about the original problem \eqref{eq:Original Linear}, one can replace $\mu$ with an available lower bound and $s_{\max}$ with an available upper bound (e.g., $s_{\max} \leq \sqrt{m}$ by part (4.)\ of \cref{assum:Assumption Linear New}) and the results of \cref{lem:Nested Complexity-Generic} will carry through. Using the stopping criterion \eqref{eq:stopping_time_map} requires an additional calculation of the gradient mapping at each iteration of the method $\calM$. The stochastic accelerated methods we consider, including the catalyst accelerated SVRG, Katyusha, and RPDG, all require periodic cyclic calculations of the full gradients and therefore this stopping criterion can be implemented without any additional complexity. (Note that, in the case of catalyst accelerated SVRG for example, we consider each ``outer loop iteration'' to be one iteration of $\calM$.)

We are now ready to specialize our complexity results to the exemplary case of using the proximal SVRG method with catalyst acceleration \cite{lin2015universal} in place of the generic method $\calM$ in \cref{algo_nested_linear}. We state the overall complexity in terms of the number of ``incremental steps,'' which we define as any calculation of: {\em (i)} the objective components $f_i(x)$ and their gradients $\nabla f_i(x)$, {\em (ii)} the penalty function values $p_\delta(a_i^Tx - b_i)$ and their gradients, and {\em (iii)} the proximal operator $\mathrm{prox}_{\psi}(\bar{x}; \alpha)$. 
In the remainder of the text, \(\tilde O\) notation is used to hide universal constants and poly-logarithmic dependencies on problem parameters and $1/\epsilon$.

\begin{proposition}\label{prop:Nested Complexity-SVRG Catalyst}
Consider applying the nested penalty method (\cref{algo_nested_linear}), with penalty parameter \(\xi \ge \bar{\xi}\) and initial smoothness parameter $\delta_0 = \frac{\xi}{\mu}\exp(-2)$, and using proximal SVRG with catalyst acceleration \cite{lin2015universal, lin2018catalyst} to solve the penalty reformulation subproblems. Let $\{\epsilon_t\}$ be an accuracy sequence for the subproblems satisfying
\begin{equation*}
\epsilon_t = \min\left\{\frac12 m\mu\delta_t^2\log^2\left(\frac{\xi s_{\max}}{\mu\delta_t}\right), {m}\xi\delta_t\log\left(\frac{\xi s_{\max}}{\mu\delta_t}\right)\right\} \text{ for } t \geq 0,
\end{equation*}
and suppose that the stopping criterion based on the gradient mapping \eqref{eq:stopping_time_map} is used with step-size \(\alpha_t = 1/(L+\frac{\xi s_{\max}^2}{4\delta_t})\), i.e., $\tau_t = \tilde{T}_{\calM, F_{\xi ,\delta_t}}(\sqrt{2\mu\epsilon_t}; \tilde{x}_{t-1}, \alpha_t)$ for all $t \geq 0$.
Then, for any desired accuracy \(\epsilon\in (0,\frac{{m}\xi }{\mu}\exp(-2)]\), the expected number of incremental steps required to satisfy 
\(\|\tilde{x}_T - x^\ast\|_2 \leq \epsilon\)
is upper bounded by
\begin{equation}\label{eq:tilde O complexity}
\tilde O\left({\ell}+\sqrt{\frac {({\ell}+{m})L} \mu}+{m}\sqrt{\frac{({\ell}+{m})\xi }{\mu\epsilon}}\right).
\end{equation}
\end{proposition}
\begin{proof}
 Suppose we use the sequence of \(\delta\): $\delta_0> \cdots >\delta_T$, where $\delta_0=\frac{\xi }{\mu}\exp(-2)$, $\delta_{t-1}/\delta_{t}=\eta>1$, \(t=1,...,T\), and $T$ is the smallest positive integer such that inequality
\[2{m}\delta_T\log\left(\frac{\xi s_{\max}}{\mu\delta_T}\right)\le\epsilon\]
holds. Since $\frac{1}{w}>2\log(\frac1w)$ for all \(w>0\), we have
\[\frac{w}{2\log(\frac1w)}\log\left(\frac{2\log(\frac1w)}w\right)\le w,\ \forall w\in(0,1).\]
Letting \(w = \frac{\mu\epsilon}{2{m}\xi s_{\max}}\) in the inequality, we see that
\[2{m}\delta\log\left(\frac{\xi s_{\max}}{\mu\delta}\right)\le \epsilon,\ \forall \delta\in\left(0,\frac{\epsilon}{4{m}}\left(\log\left(\frac{2{m}\xi s_{\max}}{\mu\epsilon}\right)\right)^{-1}\right].\]
Hence, \[\delta_T\ge\frac{\epsilon}{4{m}\eta}\left(\log\left(\frac{2{m}\xi s_{\max}}{\mu\epsilon}\right)\right)^{-1}.\]
Then, by \cref{lem:Nested Complexity-Generic}, \(\|\tilde x_T-x^*\|_2\le 2m\delta_T\log\left(\frac{\xi s_{\max}}{\mu\delta_T}\right)\leq\epsilon\) holds.

Since we start minimizing \(F_{\xi,\delta_{t+1}}(x)\),  \(t=0,...,T-1\), at \({\tilde x_t}\), by \cref{lem:Nested Complexity-Generic}
\[\begin{aligned}
&F_{\xi,\delta_{t+1}}({\tilde x_t})-F_{\xi,\delta_{t+1}}(x_{\xi,\delta_{t+1}}^*)\\ 
\le &\ (F_{\xi,\delta_{t+1}}({\tilde x_t})-F_{\xi,\delta_{t}}(x_{\xi,\delta_{t}}^*))+(F_{\xi,\delta_{t}}(x_{\xi,\delta_{t}}^*)-F_{\xi,\delta_{t+1}}(x_{\xi,\delta_{t+1}}^*))\\
\le &\ m\xi\delta_{t} \log\left(\frac{\xi s_{\max}}{\mu\delta_{t+1}}\right) + m\xi(\delta_t-\delta_{t+1}) \log2\\
\le &\  (2\eta-1)m\xi\delta_{t+1} \log\left(\frac{\xi s_{\max}}{\mu\delta_{t+1}}\right).
\end{aligned}
\]

For \(t=0,1,...,T\), we consider \(\frac{{\ell}+{m}}{\ell} f_i(x)\), \(i=1,...,{\ell}\), and \(({\ell}+{m})\xi  p_{\delta_t}(a_j^Tx-b_j)\), \(j=1,...,{m}\) to be the ${\ell}+{m}$ components. Then, the problems are \(\mu\)-strongly convex, and consist of \({\ell}\) components that are \(\frac{{\ell}+{m}}{\ell}L\)-smooth, and \({m}\) components that are \(L_t\)-smooth, respectively, where 
\[L_t=\frac{({\ell}+{m})\xi }{4\delta_t}.\]
 Then, since $\tau_t = \tilde{T}_{\calM, F_{\xi ,\delta_t}}(\sqrt{2\mu\epsilon_t}; \tilde{x}_{t-1}, \alpha_t)$, by \cref{lemma:stopping criteria:general proximal}, we have \[\tau_t\le T_{\calM,F_{\xi,\delta_t}}\left(\frac{L+\frac{\xi s_{\max}^2}{4\delta_t}}{\mu}\epsilon_t;\tilde x_{t-1}\right)\] 
 and, by \cref{lemma:Catalyst SVRG Lemma}, we have the following bounds
\begin{equation}\label{eq:Estimation in Proof}
\begin{aligned}
&\bbE\tau_0= \tilde O\left({\ell}+{m}+\sqrt{\frac{({m}+{\ell})L+
{m} L_0}{\mu}}\right)=\tilde O\left({\ell}+\sqrt{\frac{({\ell}+{m})L}{\mu}}+\sqrt{{({\ell}+{m}){m}}}\right)\\
&\bbE\tau_t= \tilde O\left({\ell}+{m}+\sqrt{\frac{({m}+{\ell})L+
{m} L_t}{\mu}}\right)\\
&= \tilde O\left({\ell}+\sqrt{\frac{({\ell}+{m})L}{\mu}}+\sqrt{\frac{({\ell}+{m}){m}\xi }{\mu\delta_t}}\right),\ t=1,...,T.
\end{aligned}    
\end{equation}
Since \(\delta_T\ge \frac{\epsilon}{4{m}\eta}\left(\log(\frac{2{m}\xi s_{\max}}{\mu\epsilon}) \right)^{-1}\),
\[\begin{aligned}
T&\le \frac{\log(\frac{4\eta {m}\xi }{\mu\epsilon}\log(\frac{2{m}\xi s_{\max}}{\mu\epsilon}))}{\log(\eta)} \cdot
\end{aligned}\]
Therefore, \(\bbE \sum_{t=1}^T\tau_t\) is bounded by 
\[\begin{aligned}
&\tilde O\left(T({\ell}+\sqrt{\frac{({\ell}+{m})L}{\mu}})+\sqrt{\frac{({\ell}+{m}){m}\xi }{\mu\delta_1}}\frac{\eta^{\frac T2}-1}{\sqrt\eta-1}\right)\\
&= \tilde O \left( {\ell}+\sqrt{\frac{({\ell}+{m})L}{\mu}} + \frac{\eta}{\sqrt{\eta}-1}{m}\sqrt{\frac{({\ell}+{m})\xi }{\mu\epsilon}}\right).
\end{aligned}\]
Then, we can take \(\eta\) to be a constant greater than 1, (or \(\eta=4\) which minimizes \(\frac{\eta}{\sqrt{\eta}-1}\)), combine it with the estimation of \(\tau_0\), and within \[
\tilde O\left({\ell}+\sqrt{\frac {({\ell}+{m})L} \mu}+{m}\sqrt{\frac{({\ell}+{m})\xi }{\mu\epsilon}}\right),
\]
expected number of incremental steps, we will reach  ${\tilde x_T}$ satisfying \(\|\tilde x_T-x^*\|_2\le \epsilon\).
\end{proof}

\begin{remark}\label{remark:epsilon_T}
It is possible to slightly improve the result of \cref{prop:Nested Complexity-SVRG Catalyst} by considering a  smaller value of the accuracy parameter $\epsilon_T$ for \emph{only} the last inner loop iteration. In particular, by the third claim in \cref{thm:estimation linear new}, since \(\delta_T\) satisfies \(2{m}\delta_T\log(\frac{\xi s_{\max}}{\mu\delta_T})\le\epsilon\), we will have \(\|x_{\xi ,\delta_T}^*-x^*\|_2\le \frac{2\epsilon}{\sqrt{m}{s_{\mathrm{pmin}}}}\). Since the dependence on \(\epsilon_T\) is logarithmic in the complexity result, if we  use
\[\epsilon_T'= \frac{ m\mu\delta_T^2}{2{s_{\mathrm{pmin}}}}\log^2\left(\frac{\xi s_{\max}}{\mu\delta_T}\right)\]
for the last inner loop, we obtain \(\|\tilde x_T-x^*\|_2\le\frac{4\epsilon}{\sqrt{m}{s_{\mathrm{pmin}}}}\) within similar bound of \(\tilde O\left({\ell}+\sqrt{\frac {({\ell}+{m})L} \mu}+{m}\sqrt{\frac{({\ell}+{m})\xi }{\mu\epsilon}}\right)\) incremental steps. Therefore, by using the accuracy $\epsilon_T'$ in the last inner loop iteration, we obtain \(\|\tilde x_T-x^*\|_2 \leq \min\{\epsilon, \frac{4\epsilon}{\sqrt{m}{s_{\mathrm{pmin}}}}\}\) within at most \(\tilde O\left({\ell}+\sqrt{\frac {({\ell}+{m})L} \mu}+{m}\sqrt{\frac{({\ell}+{m})\xi }{\mu\epsilon}}\right)\) expected number of incremental steps.
When \({s_{\mathrm{pmin}}} \ge \frac{4}{\sqrt{m}}\), the bound achieved with the new $\epsilon_T'$ value will be smaller. For example, if the rows of \(A\) are generated from i.i.d. sub-Gaussian distributions, the minimal nonzero singular value will be at the constant level, and this is expected to be the case.
\end{remark}

As a point of comparison, suppose that we instead use the accelerated proximal full-gradient method \cite{beck2009fast, nesterov2013gradient} in place of $\calM$ in \cref{algo_nested_linear}. Then, we have the following complexity result.

\begin{proposition}\label{prop:Nested Complexity-Accelerated Full Gradient}
Consider applying the nested penalty method (\cref{algo_nested_linear}), with penalty parameter \(\xi \ge \bar{\xi}\), and using the accelerated proximal full-gradient method \cite{beck2009fast, nesterov2013gradient} to solve the penalty reformulation subproblems. Let $\{\epsilon_t\}$ be an accuracy sequence for the subproblems satisfying
\begin{equation*}
0 < \epsilon_t \leq \min\left\{\frac12 m\mu\delta_t^2\log^2\left(\frac{\xi s_{\max}}{\mu\delta_t}\right), {m}\xi\delta_t\log\left(\frac{\xi s_{\max}}{\mu\delta_t}\right)\right\} \text{ for } t \geq 0,
\end{equation*}
and suppose that the stopping criterion  \eqref{eq:stopping_time_map} is used with step-size \(0 < \alpha_t \leq 1/(L+\frac{\xi s_{\max}^2}{4\delta_t})\), i.e., $\tau_t = \tilde{T}_{\calM, F_{\xi ,\delta_t}}(\sqrt{2\mu\epsilon_t}; \tilde{x}_{t-1}, \alpha_t)$ for all $t \geq 0$.
Then, for any desired accuracy \(\epsilon\in (0,\frac{{m}\xi }{\mu}\exp(-2)]\), the number of incremental steps required to satisfy 
$\|\tilde{x}_T - x^\ast\|_2 \leq \epsilon$
is upper bounded by
\[
\tilde O\left((\ell + m)\left(\sqrt{\frac L \mu} + \sqrt{\frac{{m}\xi s_{\max}^2}{\mu\epsilon}}\right)\right).
\]
\end{proposition}
\begin{proof}
The sequence of \(\delta_t\) is the same as in the proof of \cref{prop:Nested Complexity-SVRG Catalyst}. We only need to modify the estimated number of iterations \eqref{eq:Estimation in Proof}. Because \(\tilde F_{\xi,\delta_t}\) is \(\left(L+\frac{\xi s_{\max}^2}{4\delta_t}\right)\)-smooth, by Theorem 6 in \cite{nesterov2013gradient},
the number of iterations will be 
\begin{equation}\label{eq:Estimation in Proof 2}
\begin{aligned}
&\tau_0=\tilde O\left(\sqrt{\frac{L}{\mu}}+s_{\max}\right).\\
&\tau_t=\tilde O\left(\sqrt{\frac{L}{\mu}}+\sqrt{\frac{\xi s_{\max}^2 }{\mu\delta_t}}\right),\ t=1,...,T.
\end{aligned}    
\end{equation}
Similar to the proof of \cref{prop:Nested Complexity-SVRG Catalyst}, we get a solution \(\tilde x_T\) satisfying
\(\|\tilde x_T-x^*\|_2\le\epsilon\) within \(\tilde O\left(\sqrt{\frac L \mu}+\sqrt{\frac{{m}\xi s_{\max}^2}{\mu\epsilon}}\right)\) iterations, i.e., \(\tilde O\left((\ell+m)\left(\sqrt{\frac L \mu}+\sqrt{\frac{{m}\xi s_{\max}^2}{\mu\epsilon}}\right)\right)\) incremental steps.
\end{proof}

\begin{remark}\label{remark:epsilon_T_2}
Analogous to \cref{remark:epsilon_T}, for the accelerated full-gradient method, by using a slightly smaller accuracy parameter in the final inner loop iteration, we can also achieve the bound \(\|\tilde x_T-x^*\|_2\le\frac{4\epsilon}{\sqrt{m}{s_{\mathrm{pmin}}}}\) within \(\tilde O\left((\ell+m)\left(\sqrt{\frac L \mu}+\sqrt{\frac{{m}\xi s_{\max}^2}{\mu\epsilon}}\right)\right)\) expected incremental steps.
\end{remark}

The stochastic method based on the catalyst acceleration requires at least \\
\(O\left( \min\left( \sqrt{\frac{L}{\mu}}, \sqrt{{\ell}+{m}}, \sqrt{\frac{{\ell}+{m}}{m}}s_{\max} \right) \right)\) fewer evaluations on the components of the objective and the constraints and their gradients, 
compared to the accelerated full-gradient method of \cref{prop:Nested Complexity-Accelerated Full Gradient}, which indicates the advantage of application stochastic methods.

As another point of comparison, we can also obtain a complexity result for a static version of \cref{algo_nested_linear} that uses a single value of \(\delta_{\epsilon}\), which is determined based on the required accuracy \(\epsilon\).
The following corollary follows immediately from \cref{lem:Nested Complexity-Generic} and \cref{prop:Nested Complexity-SVRG Catalyst}.

\begin{corollary}\label{cor:Static Complexity-SVRG Catalyst}
For any desired accuracy \(\epsilon\in (0,\frac{{m}\xi }{\mu}\exp(-2)]\) and penalty parameter \(\xi \ge \bar{\xi}\), consider applying proximal SVRG with catalyst acceleration \cite{lin2015universal, lin2018catalyst} to solve a penalty reformulation subproblem \eqref{eq:Penalty Linear New} with penalty parameter
\[\delta_{\epsilon}=\frac{\epsilon}{4{m}}\left(\log\left(\frac{2{m}\xi s_{\max}}{\mu\epsilon}\right)\right)^{-1}.\]
Then, the expected number of incremental steps required to satisfy $\|\tilde{x} - x^\ast\|_2 \leq \min\left\{\epsilon,\frac{4\epsilon}{\sqrt{m}{s_{\mathrm{pmin}}}}\right\}$ for some $\tilde{x}$ is upper bounded by
\begin{equation} \label{eq:complexity}
\tilde O\left({\ell}+\sqrt{\frac {({\ell}+{m})L} \mu}+{m}\sqrt{\frac{({\ell}+{m})\xi }{\mu\epsilon}}\right).
\end{equation}
\end{corollary}

One can also obtain a result similar to \cref{prop:Nested Complexity-Accelerated Full Gradient} for a static version that uses the accelerated proximal gradient method. It is instructive to mention that the dynamic nested method, as presented in \cref{algo_nested_linear}, has several advantages over the static strategy of \cref{cor:Static Complexity-SVRG Catalyst}. First, in \cref{cor:Static Complexity-SVRG Catalyst}, $\delta$ depends directly on problem parameters and the desired accuracy $\epsilon > 0$, and thus \cref{cor:Static Complexity-SVRG Catalyst} only provides a solution up to but no better than an error of magnitude $\epsilon$. On the other hand, \cref{algo_nested_linear} and \cref{prop:Nested Complexity-SVRG Catalyst} provide a dynamic strategy for shrinking $\delta$ and thereby guarantee that the sequence $\{\tilde{x}_t\}$ eventually converges to $x^\ast$. Second, the nested structure has the advantage that it enables the use of the screening procedure that we develop in \cref{sec:screening} and performs better in practice as we observe in \cref{sec:numerical}. Finally, our proofs reveal that the complexity bound in \cref{prop:Nested Complexity-SVRG Catalyst} is indeed better than that of \cref{cor:Static Complexity-SVRG Catalyst} by a factor of \(O(\log\frac1{\epsilon})\), which is hidden behind the \(\tilde O\) notation. In practice, this improvement is more obvious, because the nested algorithm can address problems with relatively small condition number (due to larger \(\delta\)) compared with setting \(\delta=\tilde O(\frac{\epsilon}{m})\) instead at the beginning. 

\begin{remark}\label{remark:lower complexity bound}
When \(\epsilon>0\) is small enough, the third term in \eqref{eq:tilde O complexity} will be dominant and superior than the complexity for the deterministic method.
Ouyang and Xu \cite{ouyang2021lower} establish
lower bounds on the complexity for first-order methods for the affinely constrained convex problem:
\begin{equation}\label{eq:Ouyang Affine Constraints Problem}
\begin{aligned}
    \min_{x\in\mathbb{R}^n}\quad &f(x)\\
    \mathrm{s.t.}\quad &a_i^Tx= b_i, \  i=1,...,m.
\end{aligned}
\end{equation}
Theorem 8 in \cite{ouyang2021lower} states, for \(t=O(m)\), there are instances of \(f,a_i,b_i\) in \eqref{eq:Ouyang Affine Constraints Problem}, such that the iterate obtained from a deterministic first-order method, \(\overline{x}_t\), will have an error \(\|\overline{x}_t-x^*\|_2\ge \Omega(\frac{\|A\|_2\|\lambda^*\|_2}{\mu t})\). This result indicates a lower bound \(\Omega(\min(m^2,m\frac{\|A\|_2\|\lambda^*\|_2}{\mu\epsilon}))\) of calls on the \((a_i^Tx-b_i)\) and their gradients to have a 2-norm error \(\epsilon\) of the solution for a deterministic first-order algorithm. Compared with this result, notice for any problem with affine constraints \(Ax=b\), one can scale the rows of \(A\) and the entries of \(b\) to make the norms of rows equal to \(1\); the entries of the dual solution will be scaled respectively. Then, the third term of \eqref{eq:complexity} (for \({\ell}=O(m)\)) becomes \(\tilde O({m}\sqrt{\frac{{m}\|\lambda^*\|_{\infty}}{\mu\epsilon}})\), when \(\xi =\|\lambda^*\|_{\infty}\), which is better than the deterministic methods when \(\epsilon\le O(\frac{\|A\|_2^2\|\lambda^*\|_2^2}{{m}\mu\|\lambda^*\|_{\infty}})\), considering \(\|A\|_2\ge1\) and \(\|\lambda^*\|_2\ge\|\lambda^*\|_{\infty}\).
\end{remark}

\section{Convergence of the Duality Gap}\label{sec:duality}
In this section, we first show that given an approximate solution to the penalty reformulation problem \eqref{eq:Penalty Linear New}, we can efficiently obtain an approximate solution to the dual problem \eqref{eq:Original Dual Linear}. Then, we extend the results developed in \cref{sec:alg} to derive the complexity of obtaining both primal and dual solutions using stochastic methods, again using the proximal SVRG with catalyst acceleration as our exemplary method.

Toward establishing a result on the duality gap for \eqref{eq:Original Linear} and \eqref{eq:Original Dual Linear}, for fixed values of the parameters $\xi, \delta > 0$, we define the following Lagrangian functions:
\begin{equation}\label{eq:Linear Lagrangians}
\begin{aligned}
    &L(x,\lambda) := F(x)+\lambda^T(Ax-b),\\
    &L_{\xi ,\delta}(x,\lambda) := F(x) + \lambda^T (Ax-b) - \delta\pi_{\xi}(\lambda), 
\end{aligned}
\end{equation}
for any $x \in \bbR^n$ and $\lambda \in [0, \xi]^m$ and where \[\pi_{\xi }(\lambda) := \sum_{i=1}^{m}\lambda_i\log \lambda_i+\sum_{i=1}^{m}(\xi -\lambda_i)\log (\xi -\lambda_i)-{m}\xi \log(\xi ).\] 
Note that, when \(\lambda \in[0,\xi ]^m\), we have 
\(\pi_{\xi }(\lambda)\in [-{m}\xi \log2,0]\) and thus $L_{\xi ,\delta}(x,\lambda)-L(x,\lambda)\in[0,{m}\xi \delta\log2]$.
Following \cref{lemma:0 point}, if the penalty parameter $\xi$ is large enough, i.e., \(\xi \geq \bar{\xi}\) then there exists a dual optimal solution \(\lambda^*\) such that \(\xi \ge \bar{\xi} \ge \|\lambda^*\|_{\infty}\) and the assumption that $\lambda \in [0, \xi]^m$ is without loss of generality. Now, we define the following functions:
\begin{equation}\label{eq:Linear Dual Function}
\begin{aligned}
G(\lambda) := & \min_x L(x,\lambda) = -F^*(-A^T\lambda)-b^T\lambda,\\
G_{\xi ,\delta}(\lambda) := & \min_x L_{\xi ,\delta} (x,\lambda)=G(\lambda) - \delta\pi_{\xi }(\lambda), \\
\Phi(x): = & \max_{\lambda\in[0,\xi ]^{m}} L(x,\lambda), \\
\Phi_{\xi ,\delta}(x) := & \max_{\lambda\in[0,\xi ]^{m}} L_{\xi ,\delta}(x,\lambda),
\end{aligned}    
\end{equation}
for any $x \in \bbR^n$ and $\lambda \in [0, \xi]^m$. Notice also that, for any $x \in \bbR^n$, we have that 
\begin{equation}\label{eq:Dual Variable}
\begin{aligned}
&\arg\max_{\lambda\in[0,\xi ]^{m}} L_{\xi ,\delta}(x,\lambda) =\\ &\xi \left (\frac{\exp((a_1^Tx-b_1)/\delta)}{1+\exp((a_1^Tx-b_1)/\delta)},...,\frac{\exp((a_{m}^Tx-b_{m})/\delta)}{1+\exp((a_{m}^Tx-b_{m})/\delta)} \right )^T,
\end{aligned}    
\end{equation}
and hence 
\[\Phi_{\xi ,\delta}(x)=\max_{\in[0,\xi ]^{m}} L_{\xi ,\delta}(x,\lambda)=F_{\xi ,\delta}(x),\]
for all $x \in \bbR^n$, which follows from \eqref{eq:softplus_nesterov}.
In other words, our penalty reformulation with the softplus penalty can be considered as a version of Nesterov's smoothing technique inside the Lagrangian function. With this observation, we have the following proposition, which demonstrates how to construct an approximately optimal dual solution from an approximately optimal primal solution. Recall that $x_{\xi ,\delta}^\ast$ denotes  the unique optimal solution of \eqref{eq:Penalty Linear New} for given $\xi  \geq 0$ and $\delta \geq 0$ and $\Lambda^\ast$ denotes the set of dual optimal solutions of \eqref{eq:Original Dual Linear}.

\begin{proposition}\label{prop:Saddle Pair Linear}
Let $\epsilon > 0$ be given, and suppose that \(\xi \ge \bar{\xi} = \inf_{\lambda^\ast \in \Lambda^\ast}\|\lambda^\ast\|_\infty\) and \(0 < \delta \le \frac\epsilon {2{m}\log2}\). For a given $\hat{x}_{\xi ,\delta}$, define $\hat{\lambda}_{\xi ,\delta}$ by
\[\hat{\lambda}_{\xi ,\delta}=\xi \left (\frac{\exp((a_1^T\hat{x}_{\xi ,\delta}-b_1)/\delta)}{1+\exp((a_1^T\hat{x}_{\xi ,\delta}-b_1)/\delta)},...,\frac{\exp((a_{m}^T\hat{x}_{\xi ,\delta}-b_{m})/\delta)}{1+\exp((a_{m}^T\hat{x}_{\xi ,\delta}-b_{m})/\delta)} \right )^T.\]
If \(\hat{x}_{\xi ,\delta}\) satisfies \(\|\hat{x}_{\xi ,\delta}-x_{\xi ,\delta}^*\|_2\le \sqrt{\frac{\delta\epsilon}{m}}\min(1,\frac{4\mu\delta}{m})\), then it holds $G(\lambda^*)-G(\hat{\lambda}_{\xi ,\delta})\leq \xi \epsilon$. In addition, if \(F_{\xi ,\delta}(\hat{x}_{\xi ,\delta})-F_{\xi ,\delta}(x_{\xi ,\delta}^*)\le \xi \epsilon\), then it holds $G(\lambda^*)-G(\hat{\lambda}_{\xi ,\delta})\le F_{\xi ,0}(\hat{x}_{\xi ,\delta}) - G(\hat{\lambda}_{\xi ,\delta})\le 2\xi \epsilon$.
\end{proposition}

\begin{proof}
Letting \(\lambda_{\xi ,\delta}^*=\arg\max_{\lambda\in[0,\xi ]^{m}} G_{\xi ,\delta}(\lambda)\) and \(\lambda^*=\arg\max_{\lambda\in[0,\xi ]^{m}} G(\lambda)\), we see that \((x^*,\lambda^*)\) and \((x_{\delta}^*,\lambda_{\delta}^*)\) 
are the saddle pairs of \(L(x,\lambda)\) and \(L_{\xi ,\delta}(x,\lambda)\), respectively. 
Then, \[G(\lambda^*)\le G_{\xi ,\delta}(\lambda^*)\le G_{\xi ,\delta}(\lambda_{\xi ,\delta}^*)\le G(\lambda_{\xi ,\delta}^*)+{m}\xi \delta\log2\le G(\lambda^*)+{m}\xi \delta\log2.\]
Hence, when \(\delta \le \frac\epsilon {2{m}\log2} \), we have  \(G(\lambda^*)-G(\lambda_{\xi ,\delta}^*)\le \frac{\xi \epsilon}{2}\).

Suppose we obtain an approximate solution \(\hat{x}_{\xi ,\delta}\), and the corresponding \(\hat{\lambda}_{\xi ,\delta}=\arg\max_{\lambda\in[0,\xi ]^{m}} L_{\xi ,\delta}(\hat{x}_{\xi ,\delta},\lambda)\). Then, because \(p''_{\delta}\in[0,\frac1{4\delta}]\), we have 
\[\|\hat{\lambda}_{\xi ,\delta}-\lambda_{\xi ,\delta}^*\|_2\le \frac{\sqrt{m}}{4\delta}\|\hat{x}_{\xi ,\delta}-x_{\xi ,\delta}^*\|_2.\]
Because \( F(x)\) is \(\mu\)-strongly convex, \(F^* (y)\) is \(\frac1{\mu}\)-smooth. Since \(\|A\|_2\le \sqrt{m}\), 
\[\|\nabla G(\hat{\lambda}_{\xi ,\delta})-\nabla G(\lambda_{\xi ,\delta}^*)\|_2 \le \frac{m}{\mu}\|\hat{\lambda}_{\xi ,\delta}-\lambda_{\xi ,\delta}^*\|_2.\]
And then, 
\[\begin{aligned}
&\|\nabla G_{\xi ,\delta}(\hat{\lambda}_{\xi ,\delta})\|_2=\|\nabla G_{\xi ,\delta}(\hat{\lambda}_{\xi ,\delta})-\nabla G_{\xi ,\delta}(\lambda_{\xi ,\delta}^*)\|_2\\
\le\ & \|\nabla G(\hat{\lambda}_{\xi ,\delta})-\nabla G(\lambda_{\xi ,\delta}^*)\|_2+\delta\|\nabla \pi_{\xi }(\hat{\lambda}_{\xi ,\delta}) -\nabla \pi_{\xi }(\lambda_{\xi ,\delta}^*) \|_2\\
\le\ & \frac{m}{\mu}\|\hat{\lambda}_{\xi ,\delta}-\lambda_{\xi ,\delta}^*\|_2 + \|A(\hat{x}_{\xi ,\delta}-x_{\xi ,\delta}^*)\|_2\\
\le\ & \left (\frac{{m}\sqrt{m}}{4\mu\delta} + \sqrt{m} \right )\|\hat{x}_{\xi ,\delta}-x_{\xi ,\delta}^*\|_2.
\end{aligned}\]
Since \(G(\lambda)\) is concave, \(G_{\xi ,\delta}(\lambda)\) is \(\frac{4\delta}{\xi }\)-strongly concave. Therefore,
\[G_{\xi ,\delta}(\lambda_{\xi ,\delta}^*)-G_{\xi ,\delta}(\hat{\lambda}_{\xi ,\delta}) \le \frac{\xi }{8\delta }\|\nabla G_{\xi ,\delta}(\hat{\lambda}_{\xi ,\delta})\|_2^2,\]
which indicates that when \(\|\hat{x}_{\xi ,\delta}-x_{\xi ,\delta}^*\|_2\le \sqrt{\frac{\delta\epsilon}{m}}\min(1,\frac{4\mu\delta}{m})\),
\[G_{\xi ,\delta}(\lambda_{\xi ,\delta}^*)-G_{\xi ,\delta}(\hat{\lambda}_{\xi ,\delta})\le \frac{\xi \epsilon}2,\]
and then, when additionally \(\delta \le \frac\epsilon {2{m}\log2} \),
\[\begin{aligned}
G(\lambda^*) &\ge G(\hat{\lambda}_{\xi ,\delta}) \ge G_{\xi ,\delta}(\hat{\lambda}_{\xi ,\delta})- \frac{\xi \epsilon}2 \\&\ge G_{\xi ,\delta}(\lambda_{\xi ,\delta}^*) - \xi \epsilon \ge G_{\xi ,\delta}(\lambda^*) - \xi \epsilon\ge G(\lambda^*)- \xi \epsilon.
\end{aligned}\]
In other words, \(\lambda_{\xi , \delta,\epsilon}\) is a \((\xi \epsilon)\)-solution to the dual problem \eqref{eq:Original Dual Linear}.

Now when \(F_{\xi ,\delta}(\hat{x}_{\xi ,\delta})-F_{\xi ,\delta}(x_{\xi ,\delta}^*)\le \xi \epsilon\),
\[\begin{aligned}
G(\lambda^*)&=F(x^*)=F_{\xi ,0}(x^*)\le F_{\xi ,0}(\hat{x}_{\xi ,\delta})\le F_{\xi ,0}(\hat{x}_{\xi ,\delta})\\ 
&\le F_{\xi ,\delta}(\hat{x}_{\xi ,\delta}) \le G_{\xi ,\delta}(\lambda_{\xi ,\delta}^*) + \xi \epsilon \le G(\lambda_{\xi_{\infty,\delta,\epsilon}}) + 2\xi \epsilon,
\end{aligned}\]
and then \(G(\lambda^*)-G(\hat{\lambda}_{\xi ,\delta})\le F_{\xi ,0}(\hat{x}_{\xi ,\delta}) - G(\hat{\lambda}_{\xi ,\delta})\le 2\xi \epsilon\).
\end{proof}

Notice that the gap \(F_{\xi ,0}(\hat{x}_{\xi ,\delta}) - G(\hat{\lambda}_{\xi ,\delta})\), which we bound in \cref{prop:Saddle Pair Linear}, is the duality gap between the penalized problem \(\min_{x} F_{\xi ,0}(x)\),
and its dual 
\[\begin{aligned}
    \max_{\lambda}\quad &G(\lambda) := -F^*(-A^T\lambda)-b^T\lambda,\\
    \mathrm{s.t.}\quad &0\le\lambda\le\xi ,
\end{aligned}\]
which have the same optimal costs with the original primal and dual problems. Hence, we report it as the duality gap in the numerical experiments in \Cref{sec:numerical}.
In addition, since the duality gap \(F_{\xi,\delta}(\hat x_{\xi,\delta})-G_{\xi,\delta}(\hat \lambda_{\xi,\delta})\) upper bounds \(F_{\xi,\delta}(\hat x_{\xi,\delta})-F_{\xi,\delta}^*\), we can use this gap as the stopping criterion (e.g., instead of the norm of the gradient mapping) in Algorithm \ref{algo_nested_linear}. One can then develop results analogous to \cref{lem:Nested Complexity-Generic} and \cref{prop:Nested Complexity-SVRG Catalyst} that use this duality gap in place of the norm of the gradient mapping.

We are now ready to state our main theorem on the primal and dual convergence of the nested penalty method (\cref{algo_nested_linear}) using proximal SVRG with catalyst acceleration \cite{lin2015universal, lin2018catalyst}, which combines the results of \cref{prop:Nested Complexity-SVRG Catalyst} and \cref{prop:Saddle Pair Linear}. Note that, given the sequence $\{\tilde{x}_t\}$ output by \cref{algo_nested_linear}, we define a sequence of dual solutions $\{\tilde{\lambda}_t\}$ by
\[\tilde{\lambda}_t := \xi \left (\frac{\exp((a_1^T\tilde{x}_t-b_1)/\delta_t)}{1+\exp((a_1^T\tilde{x}_t-b_1)/\delta_t)},...,\frac{\exp((a_{m}^T\tilde{x}_t-b_{m})/\delta_t)}{1+\exp((a_{m}^T\tilde{x}_t-b_{m})/\delta_t)} \right )^T.\]

\begin{theorem}\label{Thm: Main Linear New}
Consider applying the nested penalty method (\cref{algo_nested_linear}), with penalty parameter \(\xi \ge \bar{\xi}\), and using proximal SVRG with catalyst acceleration \cite{lin2015universal, lin2018catalyst} to solve the penalty reformulation subproblems. Let $\{\epsilon_t\}$ be an accuracy sequence for the subproblems satisfying
\begin{equation*}
\epsilon_t = \min\left\{\frac12 m\mu\delta_t^2\log^2\left(\frac{\xi s_{\max}}{\mu\delta_t}\right), {m}\xi\delta_t\log\left(\frac{\xi s_{\max}}{\mu\delta_t}\right)\right\} \text{ for } t \geq 0,
\end{equation*}
and suppose that the stopping criterion  \eqref{eq:stopping_time_map} is used with step-size \(\alpha_t = 1/(L+\frac{\xi s_{\max}^2}{4\delta_t})\), i.e., $\tau_t = \tilde{T}_{\calM, F_{\xi ,\delta_t}}(\sqrt{2\mu\epsilon_t}; \tilde{x}_{t-1}, \alpha_t)$ for all $t \geq 0$.
Then, for any desired accuracy \(\epsilon\in (0,\frac{{m}\xi }{\mu}\exp(-2)]\), the expected number of incremental steps required to satisfy 
\(\|\tilde{x}_T - x^\ast\|_2 \leq \epsilon\) and \(G(\lambda^*)-G(\tilde{\lambda}_T)\le F_{\xi,0}(\tilde x_T) - G(\tilde{\lambda}_T)\le \xi\epsilon\)
is upper bounded by
\[
\tilde O\left({\ell}+\sqrt{\frac {({\ell}+{m})L} \mu}+{m}\sqrt{\frac{({\ell}+{m})\xi }{\mu\epsilon}}\right).
\]
\end{theorem}
\begin{remark}
Analogous to \cref{remark:epsilon_T}, by using a  smaller accuracy parameter in the final inner loop iteration, in \cref{Thm: Main Linear New}, one can also achieve the bound \(\|\tilde x_T-x^*\|_2\le\frac{4\epsilon}{\sqrt{m}{s_{\mathrm{pmin}}}}\) within \(\tilde O\left((\ell+m)\left(\sqrt{\frac L \mu}+\sqrt{\frac{{m}\xi s_{\max}^2}{\mu\epsilon}}\right)\right)\) expected incremental steps.
\end{remark}

\section{Screening Procedure for the Nested Penalty Method}\label{sec:screening}

In this section, we describe an additional enhancement to the dynamic nested penalty method (\cref{algo_nested_linear}) that uses a \emph{safe} screening procedure to remove constraints that are not tight. In particular, the KKT conditions of problem \eqref{eq:Original Linear} imply that if the \(i\)-th constraint is not tight, i.e., it is not active at \(x^*\), \(a_i^Tx^*-b_i < 0\), then the corresponding dual variable \(\lambda_i^*=0\) and removing this constraint from \eqref{eq:Original Linear} will not change the optimal solution \(x^*\). Recall that \cref{prop:estimation linear new z infty m} provides a bound on the norm of $z_{\xi ,\delta}=A(x_{\xi ,\delta}^*-x_{\xi, 0}^*)$, i.e., the distance between the slack variables at $x_{\xi ,\delta}^*$ and $x_{\xi, 0}^*$. Based on this bound, if $\xi \geq \bar{\xi}$ and if the slack \(b_i - a_i^Tx_{\xi ,\delta}^*\) 
is large enough, then one can safely conclude that \(a_i^Tx^*-b_i < 0\) and hence the \(i\)-th constraint can be dropped from the problem. Once can also apply a similar procedure to the approximate solutions $\tilde{x}_t$ to the penalty reformulation problems \eqref{eq:Penalty Linear New} computed at each outer loop iteration of the nested method (\cref{algo_nested_linear}). We refer to dropping such constraints as a screening procedure, and we prove that it is \emph{safe}, i.e., does not remove constraints that are tight at an optimal solution, by combining a result similar to \cref{prop:estimation linear new z infty m} with a bound on the distance between the slack variables at the approximate penalty solution $\tilde{x}_t$ and the exact solution $x_{\xi ,\delta_t}^*$.
Note that safe screening procedures to eliminate redundant variables and constraints are prevalent in machine learning literature (see, e.g., \cite{ghaoui2010safe,AG:screening}). 

Let us now formally introduce these ideas by considering the following version of the original problem \eqref{eq:Original Linear} based on a subset of the constraint indices \(\hat{S} \subseteq S^F := \{1, \ldots, m\}\):
\begin{equation}\label{eq:Original Linear Subset}
\begin{aligned}
\min_{x\in\mathbb{R}^n}\quad &F(x)=\frac1{\ell}\sum_{i=1}^{\ell}f_i(x)+\psi(x)\\
\mathrm{s.t.}\quad &a_i^Tx\le b_i, \  i\in \hat S,
\end{aligned}
\end{equation}
as well as the corresponding the penalty reformulation:
\begin{equation}\label{eq:Penalty Linear New Subset}
    \min_{x \in \bbR^n} F_{\hat S, \xi ,\delta}(x):=F(x)+\xi  \sum_{i\in \hat S} p_{\delta}(a_i^Tx-b_i).
\end{equation}
Let \(x_{\hat S,\xi,\delta}^*\) denote the unique optimal solution of \eqref{eq:Penalty Linear New Subset} for a given $\xi \geq 0$, $\delta \geq 0$. Furthermore, let \(S^{A} := \{i\in S^F:a_i^Tx^*-b_i=0\}\) denote the set of indices of active (tight) constraints at the optimal solution $x^\ast$ of \eqref{eq:Original Linear}. Also, for \(\Delta \ge 0\), we let \({S_{\Delta}}\defeq\{i\in S^F:a_i^Tx^*-b_i\in [-\Delta,0]\}\) and \({m_{\Delta}}\defeq\mathrm{card}({S_{\Delta}})\). Notice, \( S^A \subseteq {S_{\Delta}} \subseteq S^F\) and \(m^A\defeq\mathrm{card}(S^A)\le {m_{\Delta}}\le m\) for all \(\Delta\ge 0\). Finally, let \(A_S\) denote the submatrix consisting of the rows of \(A\) with indices in \(S\).
Then, the Lemma below follows from applying \cref{lemma:0 point} on \eqref{eq:Original Linear Subset} and \eqref{eq:Penalty Linear New Subset}.
\begin{lemma}\label{lemma:0 Point Subset}
If \(\xi \ge \bar{\xi}\) and \(S^{A}\subseteq \hat S \subseteq S^{F}\), then it holds that \(x_{\hat S, \xi ,0}^*=x^*\).
\end{lemma}
Similarly, we obtain \cref{prop:estimation linear new z infty m Subset} as a corollary of \cref{prop:estimation linear new z infty m}.
\begin{proposition}\label{prop:estimation linear new z infty m Subset}
For \(\xi >0\) and \(\delta\in [0,\frac{s_{\max}^2(A_S)\xi }{\mu}\exp(-2)]\), it holds that 
\[\|z_{S, \xi ,\delta}\|_2=\|A_S(x_{S,\xi,\delta}^*-x_{S,\xi,0}^*)\|_2\le \sqrt{\mathrm{card}(S)}\delta\log\left(\frac{s_{\max}^2(A_S)\xi }{\mu\delta}\right) \cdot\]
\end{proposition}
\cref{algo_nested_linear_screening} presents the nested penalty method that includes the screening procedure at every outer loop iteration. \cref{lem:Nested Complexity-Generic subset} presents a convergence results for \cref{algo_nested_linear_screening} and verifies that the screening procedure safely removes only constraints that are guaranteed to not be tight at the optimal solution $x^\ast$.

\begin{algorithm}[ht]
\caption{Generic Nested Penalty Method with Screening Procedure}
\label{algo_nested_linear_screening}
\begin{algorithmic}
    \STATE {\bf Parameters:}  Penalty parameter $\xi > 0$, smoothness scaling parameter \(\eta>1\), and generic unconstrained convex optimization method $\calM$.
    \STATE {\bf Initialize:} $\hat{x}_{-1} \in \bbR^n$, $\delta_{0} > 0$, \(\hat S_{0}=S^F\), \(\hat m_{0} = m\).
    \FOR{$t = 0, 1, \ldots, T$}
    \STATE 1. 
    Apply method $\calM$ on the penalty reformulation problem \eqref{eq:Penalty Linear New Subset},
    initialized at $\tilde{x}_{t-1}$ and for $\tau_t \geq 1$ total iterations, to obtain an approximate solution
    \begin{equation*}
        \hat{x}_t \approx \arg\min_{x \in \bbR^n} \left\{F_{\hat S_t,\xi,\delta_t}(x) := F(x) + \xi\sum_{i\in \hat S_t} p_{\delta_t}(a_i^Tx-b_i)\right\}.
    \end{equation*}
    \STATE 2. Choose step-size $\alpha_t > 0$ and set $\tilde{x}_t \gets \mathrm{prox}_{\psi}(\hat{x}_t - \alpha_t\nabla\tilde{F}_{\hat S_t,\xi,\delta_t}(\hat{x}_t); \alpha_t)$.
    \STATE 3. Screening procedure: set \(\hat S_{t+1} \gets \{i \in \hat{S}_t : a_i^T \tilde x_{t}-b_i \geq -2\sqrt{\hat m_t}\delta_t \log (\frac{\hat m_t\xi}{\mu\delta})\}\) and update \(\hat m_{t+1}\gets \mathrm{card}(\hat S_{t+1})\).
    \STATE 4. Update $\delta_{t+1} \gets \delta_{t}/\eta$.
    \ENDFOR
\end{algorithmic}
\end{algorithm}

\begin{lemma}\label{lem:Nested Complexity-Generic subset}
Consider applying the generic nested penalty method (\cref{algo_nested_linear_screening}) using penalty parameter \(\xi \ge \bar{\xi}\) and initial smoothness parameter $\delta_0 \leq \frac{\xi}{\mu}\exp(-2)$. Let $\{\epsilon_t\}$ be an accuracy sequence for the penalty reformulation subproblems satisfying
\begin{equation*}
0 < \epsilon_t \leq \frac12 \hat m_t\mu\delta_t^2\log^2\left(\frac{\xi s_{\max}}{\mu\delta_t}\right) \text{ for } t \geq 0,
\end{equation*}
and suppose that the stopping criterion  \eqref{eq:stopping_time_map} is used with step-size \(0 < \alpha_t \leq 1/(L+\frac{\xi s_{\max}^2}{4\delta_t})\), i.e., $\tau_t = \tilde{T}_{\calM, F_{\xi ,\delta_t}}(\sqrt{2\mu\epsilon_t}; \tilde{x}_{t-1}, \alpha_t)$ for all $t \geq 0$.
Then, for any $t \geq 0$, it holds that:
\begin{equation}\label{eq:generic_convergence subset}
\begin{aligned}
\|\tilde{x}_t - x^\ast\|_2 &\leq \left(\min\left\{\hat m_t,\frac{4\sqrt{\hat m_t}}{{s_{\mathrm{pmin}}}}\right\}+\sqrt{\hat m_t}\right) \delta_t \log\left(\frac{\xi s_{\max}}{\mu\delta_t}\right)\\&=\frac{\left(\min\left\{\hat m_t,\frac{4\sqrt{\hat m_t}}{{s_{\mathrm{pmin}}}}\right\}+\sqrt{\hat m_t}\right)\delta_0}{\eta^t}\left[t\log(\eta) + \log\left(\frac{\xi s_{\max}}{\mu\delta_0}\right)\right].
\end{aligned}
\end{equation}
In addition, it holds that \(S^A\subseteq \hat S_{t+1} \subseteq S_{\hat \Delta_t} \subseteq S_{\Delta_t},\)
where
\[ 
\hat \Delta_{t} = 4\sqrt{\hat m_{t}}\delta_t\log\left(\frac{\hat m_{t}\xi}{\mu\delta_t}\right),\ \Delta_{t}= 4\sqrt{m}\delta_t\log\left(\frac{m\xi}{\mu\delta_t}\right).\]
\end{lemma}

\ignore{
\begin{lemma}\label{lemma:stopping criteria nested Subset}
Let \(\{\hat x_k^{(t)}\}\) be the sequence generated by \(\calM\) applied to minimize \(F_{\hat S_t, \xi, \delta_t}\) in \cref{algo_nested_linear_screening}. Let \[\epsilon_t=\min\left\{\frac12 \hat m_t\mu\delta_t^2\log^2\left(\frac{\xi \hat m_t}{\mu\delta_t}\right), {\hat m_t}\xi\delta_t\log\left(\frac{\xi \hat m_t}{\mu\delta_t}\right)\right\} \text{ for } t \geq 0.\]
Let \(\tau_t := \inf\{k \geq 0 : \|g_{\alpha_t,F_{\hat S_t,\xi,\delta_t}}(\hat x_k^{(t)})\|_2\le \sqrt{2\mu\epsilon_t}\}\), where \(\alpha_t = 1/(L+\frac{\xi \hat m_t}{4\delta_t})\). Then, let \(\tilde x_t = \mathrm{prox}_{\alpha\psi}(\hat x_{\tau_t}^{(t)}-\alpha\nabla f_{\hat S_t,\xi,\delta}(\hat x_{\tau_t}^{(t)}) )\), we will have: 1, \[F_{\hat S_t,\xi,\delta_t}(\tilde x_t)-F_{\hat S_t,\xi,\delta}^*\le \epsilon_t,\] and \(\tau_t\le T_{\calM, F_{S_t,\xi,\delta_t}}(\frac{\mu}{L+\frac{\xi \hat m_t}{4\delta_t}}\epsilon_t; \hat x_{t-1})\); and 2, \[\hat S_{t+1}=\{i\in \hat S_t : a_i^T \hat x_{t}-b_i \ge -2\sqrt{\hat m_t}\delta_t \log (\frac{\hat m_t\xi}{\mu\delta}) \}\] satisfies 
\(S^A\subset \hat S_{t+1} \subset S_{\hat \Delta_t} \subset S_{\Delta_t},\)
where
\[ 
\hat \Delta_{t} = 4\sqrt{\hat m_{t}}\delta_t\log\left(\frac{\hat m_{t}\xi}{\mu\delta_t}\right),\ \Delta_{t}= 4\sqrt{m}\delta_t\log\left(\frac{m\xi}{\mu\delta_t}\right).\]
\end{lemma}
}
\begin{proof}
We only need to prove the second part as the first part is completely analogous to the proof of \cref{lem:Nested Complexity-Generic}. Suppose \(S^A \subseteq \hat S_t \), then by \cref{prop:estimation linear new z infty m Subset}, for all \(i\in \hat S_t\), we have
\[
|a_i^T x_{\hat S_t,\xi,\delta_t}^*-a_i^Tx^*| \le \|A_{\hat S_t}(x_{\hat S_t,\xi,\delta_t}^*-x^*)\|_2 \le \sqrt{\hat m_t}\delta_t\log\left(\frac{s_{\max}^2(A_{\hat S_t})\xi }{\mu\delta_t}\right).\]
Also, by \(\mu\)-strongly convexity of \(F_{\hat S_t,\xi,\delta_t}\), \(\|\tilde x_t-x_{\hat S_t \xi,\delta_t}^*\|_2\le \sqrt{\hat m_t}\delta_t\log\left(\frac{ \hat m_t\xi}{\mu\delta_t}\right)\). And then, since \(\|a_i\|_2=1\) (\cref{assum:Assumption Linear New}), 
\[|a_i^T \tilde x_t-a_i^Tx^*|\le 2\sqrt{\hat m_t}\delta_t\log\left(\frac{\hat m_t\xi }{\mu\delta_t}\right),\]
which means \(S^A\subseteq \hat S_{t+1} \subseteq S_{\hat \Delta_t} \subseteq S_{\Delta_t}\). By induction, since \(S^A \subseteq \hat S_0=S^F\), the result holds for all \(t=0,...,T-1\).
\end{proof}
As mentioned, the first part of \cref{lem:Nested Complexity-Generic subset} states the convergence of the nested method, as a corollary of \cref{lem:Nested Complexity-Generic}. The second part shows the safety of the screening procedure, as a result of the choice of \(\epsilon_t\) and \cref{prop:estimation linear new z infty m Subset}.

Similar to \cref{prop:Saddle Pair Linear}, we can also obtain a dual solution when solving the penalty reformulation \eqref{eq:Penalty Linear New Subset}.
\begin{proposition}\label{prop:Saddle Pair Linear Subset}
Let $\epsilon > 0$ be given, and suppose that \(\xi \ge \bar{\xi} = \inf_{\lambda^\ast \in \Lambda^\ast}\|\lambda^\ast\|_\infty\), \(S^{A}\subseteq \hat S \subseteq S^{F}\), \(\hat m =\mathrm{card}(\hat S)\), and \(0 < \delta \le \frac\epsilon {2{\hat m}\log2}\). For a given $\hat{x}_{\hat S,\xi ,\delta}$, define $\hat{\lambda}_{\hat S,\xi ,\delta}$ by
\[\hat \lambda_{\hat S,\xi,\delta,i}=\left\{
\begin{aligned}
&\ \xi \frac{\exp((a_i^T\hat{x}_{\hat S,\xi,\delta}-b_i)/\delta)}{1+\exp((a_i^T\hat{x}_{\hat S,\xi,\delta}-b_i)/\delta)},\quad &i\in \hat S\\
&\ 0, \quad &i\in S^F\backslash \hat S.
\end{aligned}
\right.\]
If \(\hat{x}_{\hat S,\xi ,\delta}\) satisfies \(\|\hat{x}_{\hat S,\xi ,\delta}-x_{\hat S,\xi ,\delta}^*\|_2\le \sqrt{\frac{\delta\epsilon}{\hat m}}\min(1,\frac{4\mu\delta}{\hat m})\), then it holds that $G(\lambda^*)-G(\hat{\lambda}_{\hat S,\xi ,\delta})\leq \xi \epsilon$. In addition, if \(F_{\hat S,\xi ,\delta}(\hat{x}_{\hat S,\xi ,\delta})-F_{\hat S,\xi ,\delta}(x_{\hat S,\xi ,\delta}^*)\le \xi \epsilon\), then it holds that $G(\lambda^*)-G(\hat{\lambda}_{\hat S,\xi ,\delta})\le F_{\hat S,\xi ,0}(\hat{x}_{\hat S,\xi ,\delta}) - G(\hat{\lambda}_{\hat S,\xi ,\delta})\le 2\xi \epsilon$.
\end{proposition}
\ignore{
\begin{proposition}\label{prop:Saddle Pair Linear Subset}
Suppose that \(\xi \ge \bar{\xi}\), \(S^{A}\subseteq \hat S \subseteq S^{F}\), \(\hat m =\mathrm{card}(\hat S)\), and \(\delta \le \frac\epsilon {2{\hat m}\log2}\). If \(\hat{x}_{\hat S,\xi,\delta}\) satisfy \(\|\hat{x}_{\hat S,\xi ,\delta}-x_{\hat S,\xi,\delta}^*\|_2\le \sqrt{\frac{\delta\epsilon}{\hat m}}\min(1,\frac{4\mu\delta}{\hat m})\). Then, if we take \(\hat \lambda_{\hat S,\xi,\delta}\) to be 
\[\hat \lambda_{\hat S,\xi,\delta,i}=\left\{
\begin{aligned}
&\ \xi \frac{\exp((a_i^T\hat{x}_{\hat S,\xi,\delta}-b_i)/\delta)}{1+\exp((a_i^T\hat{x}_{\hat S,\xi,\delta}-b_i)/\delta)},\quad &i\in \hat S\\
&\ 0, \quad &i\in S^F\backslash \hat S,
\end{aligned}
\right.\]
then it satisfies \[G(\lambda^*)-G(\hat{\lambda}_{\hat S, \xi ,\delta})\leq \xi \epsilon.\] 
\end{proposition}}
Let \({\varsigma} := \min_{i\in S^F\backslash S^A} (b_i-a_i^Tx^*)\) denote the smallest slackness of inactive constraints at \(x^*\), which naturally arises as part of the complexity analysis of our screening procedure. In particular, we obtain the following result by combining \cref{prop:Nested Complexity-SVRG Catalyst}, Lemma \ref{lem:Nested Complexity-Generic subset}, and Proposition \ref{prop:Saddle Pair Linear Subset}.

\begin{proposition}\label{Thm: Main Linear New subset}
Consider applying the nested penalty method with screening procedure (\cref{algo_nested_linear_screening}), with penalty parameter \(\xi \ge \bar{\xi}\) and initial smoothness parameter $\delta_0 = \frac{\xi}{\mu}\exp(-2)$, and using proximal SVRG with catalyst acceleration \cite{lin2015universal, lin2018catalyst} to solve the penalty reformulation subproblems. Let $\{\epsilon_t\}$ be an accuracy sequence for the subproblems satisfying
\begin{equation*}
\epsilon_t = \min\left\{\frac12 \hat m_t\mu\delta_t^2\log^2\left(\frac{\xi s_{\max}}{\mu\delta_t}\right), {\hat m_t}\xi\delta_t\log\left(\frac{\xi s_{\max}}{\mu\delta_t}\right)\right\} \text{ for } t \geq 0,
\end{equation*}
and suppose that the stopping criterion based on the gradient mapping \eqref{eq:stopping_time_map} is used with step-size \(\alpha_t = 1/(L+\frac{\xi s_{\max}^2}{4\delta_t})\), i.e., $\tau_t = \tilde{T}_{\calM, F_{\xi ,\delta_t}}(\sqrt{2\mu\epsilon_t}; \tilde{x}_{t-1}, \alpha_t)$ for all $t \geq 0$.
Then, for any desired accuracy \(\epsilon\in (0,\frac{{m}\xi }{\mu}\exp(-2)]\), the expected number of incremental steps required to satisfy 
\(\|\tilde{x}_T - x^\ast\|_2 \leq \epsilon\) and \(G(\lambda^*)-G(\tilde{\lambda}_T)\le F_{S^A,\xi,0}(\tilde x_T) - G(\tilde{\lambda}_T)\le \xi\epsilon\)
is upper bounded by
\[
\tilde O\left({\ell}+\sqrt{\frac {({\ell}+{m})L} \mu}+\sqrt{\frac{m({\ell}+{m})\xi }{\mu\varsigma}}+{m^A}\sqrt{\frac{({\ell}+{m^A})\xi }{\mu\epsilon}}\right).
\]
\end{proposition}
\begin{proof}
Suppose we use the sequence of \(\delta\): $\delta_0>...>\delta_T$, where $\delta_0=\frac{\xi }{\mu}\exp(-2)$, $\delta_{t-1}/\delta_{t}=\eta>1$, \(t=1,...,T\), and $T$ is the smallest positive integer such that the following inequality holds
\[2{\hat m_T}\delta_T\log\left(\frac{\xi \hat m_T}{\mu\delta_T}\right)\le\epsilon.\] Then, the estimation of the complexity is similar with the proof for \cref{prop:Nested Complexity-SVRG Catalyst}. Notice 
\begin{equation}\label{eq:Estimation in Proof cor final}
\begin{aligned}
&\bbE\tau_0= \tilde O\left({\ell}+{m}+\sqrt{\frac{({m}+{\ell})L+
{m} L_0}{\mu}}\right)=\tilde O\left({\ell}+\sqrt{\frac{({\ell}+{m})L}{\mu}}+\sqrt{{({\ell}+{m}){m}}}\right).\\
&\bbE\tau_t= \tilde O\left({\ell}+{\hat m_t}+\sqrt{\frac{({\hat m_t}+{\ell})L+
{\hat m_t} L_t}{\mu}}\right)\\
&= \tilde O\left({\ell}+\sqrt{\frac{({\ell}+{\hat m_t})L}{\mu}}+\sqrt{\frac{({\ell}+{\hat m_t}){\hat m_t}\xi }{\mu\delta_t}}\right),\ t=1,...,T.
\end{aligned}.    
\end{equation}
By \cref{lem:Nested Complexity-Generic subset}, when \(4\sqrt{m}\delta_t\log\left(\frac{m\xi}{\mu\delta_t}\right) \le \varsigma \), \(\hat S_{t+1}=S^A\). Hence, \[\bbE\tau_t\le \tilde O\left({\ell}+\sqrt{\frac {({\ell}+{m})L} \mu}+\max\left(\sqrt{\frac{m({\ell}+{m})\xi }{\mu\varsigma}},{m^A}\sqrt{\frac{({\ell}+{m^A})\xi }{\mu\epsilon}}\right)\right).\]
The total complexity follows from the summation of \(T\le O(\log(\frac{m\xi}{\mu\epsilon}))\) estimations.
\end{proof}

As compared to \cref{prop:Nested Complexity-SVRG Catalyst}, \cref{Thm: Main Linear New subset} introduces an additional term, which is $\tilde O\left(\sqrt{\frac{m({\ell}+{m})\xi }{\mu\varsigma}}\right)$, to the complexity bound, and which naturally depends inversely proportional to the square root of $\varsigma$. On the other hand, the last term of the complexity bound of \cref{prop:Nested Complexity-SVRG Catalyst} is $\tilde O\left({m}\sqrt{\frac{({\ell}+{m})\xi }{\mu\epsilon}}\right)$, whereas \cref{Thm: Main Linear New subset} replaces $m$ in this term with $m^A$, the number of \emph{active} constraints at $x^\ast$, which can be much smaller. In particular, when \(t\) goes to infinity, \(\hat S_{t}\) will converge to \(S^A\). In other words, in the long run, only the tight constraints will be maintained in the penalty reformulations in \cref{algo_nested_linear_screening}. Furthermore, if \(m^A\ne 0\), the $\tilde O\left({m^A}\sqrt{\frac{({\ell}+{m^A})\xi }{\mu\epsilon}}\right)$ term in \cref{Thm: Main Linear New subset} will dominate the complexity for small enough $\epsilon$, indicating a substantial improvement due to the incorporation of the screening procedure, especially when $m^A \ll m$.

\section{Numerical Results}
\label{sec:numerical}
In this section, we present the our numerical experiments performed to test the empirical effectiveness and safety of the proposed algorithms. The first problem used for the experiments
is quadratic programming, i.e., the minimization of a convex quadratic function over linear inequality constraints; the second one is a support vector machine problem. We mainly focus on three algorithms for solving the problems: \texttt{SASC-SGD}: the SASC algorithm in \cite{fercoq2019almost}; \texttt{Nested-SGD}: \cref{algo_nested_linear} with SGD for solving the sub-problems; \texttt{Nested-SGDM}: \cref{algo_nested_linear} with the momentum method (see \eqref{Nesterov Momentum} in \cref{sec:details of numerical}) for solving the sub-problems. 
For the simplicity of implementation, we use SGD instead of SVRG for most of the experiments, and the momentum method instead of the catalyst accelerated SVRG. We will also compare their performance with the parameter updating scheme in \cite{nedic2020convergence}, and the static penalty method \texttt{Static-SGDM}. We do not to include the comparison with algorithms in \cite{mishchenko2018stochastic} because it has a similar penalty structure with \cite{fercoq2019almost}, while the latter contains more information about implementation. For the second experiment, we also include comparisons with \texttt{Nested-SVRG}: \cref{algo_nested_linear} with SVRG solving the sub-problems, and \texttt{Nested-SVRG-Screening}: \cref{algo_nested_linear_screening} with SVRG solving the sub-problems, to verify the safety and efficiency of the screening procedure in \cref{sec:screening}. Further details about experiments are given in \cref{sec:details of numerical}.
\subsection{Quadratic Programming}
\label{subsec:qp}
We consider the following QP problem 
\begin{equation}\label{eq:New QP}
\begin{aligned}
\min_{x\in\mathbb{R}^n}\quad &\frac{1}{2{\ell}}\sum_{i=1}^{\ell}(\phi_i^Tx-y_i)^2+\frac w2 \|x\|_2^2\\
\mathrm{s.t.}\quad &a_j^Tx\le b_j,\ j=1,...,{m},
\end{aligned}
\end{equation}
where \(\Phi=(\phi_1,...,\phi_{\ell})^T\in\bbR^{{\ell}\times n}\), \(y=(y_1,...,y_{\ell})^T\in\mathbb{R}^{\ell}\), \(A\in(a_1,...,a_{m})^T\in\mathbb{R}^{{m}\times n}\), and \(b=(b_1,...,b_{m})^T\in\mathbb{R}^{m}\) are generated randomly. Namely, the entries \(y\), \(b_0\) are all generated i.i.d. from the standard normal distribution. We then let \(b\) be the absolute value vector of \(b_0\) to ensure the feasible region is nonempty. The entries of \(A\) are drawn from i.i.d. Normal(0,1), and then normalized so that \(\|a_i\|_2=1\), \(i=1,...,{m}\). We let \({\ell}={m}=n=100\), and the regularization parameter \(w=0.1\). All algorithms are initialized at point \(x_0=0\).

For this experiment, we compare the relative error of the solutions found by the algorithms, i.e., \(\frac{\|x-x_{\text{cvx}}\|_2}{\|x_{\text{cvx}}\|_2}\), where \(x_{\text{cvx}}\) is the solution obtained from \texttt{CVX} \cite{cvx,gb08} calling \texttt{gurobi} \cite{gurobi}. We first select the best parameters for the three methods based on 20 simulations (see \cref{sec:details of numerical}). \cref{Fig_qp_bp_nested} shows the box plot of the relative errors for the three methods after 1e7 iterations. We could see that \texttt{Nested-SGDM} has the fastest convergence after 1e7 iterations, and then it's \texttt{Nested-SGDM}. \cref{Fig_qp_bp_static} shows a similar box plot for \texttt{Static-SGDM}, where \(\delta\) is chosen as the initial, the last, and their geometric average in \texttt{Nested-SGDM}.

\begin{figure}[ht]
\begin{minipage}[ht]{0.5\linewidth}
\centering
\includegraphics[width=1\textwidth]{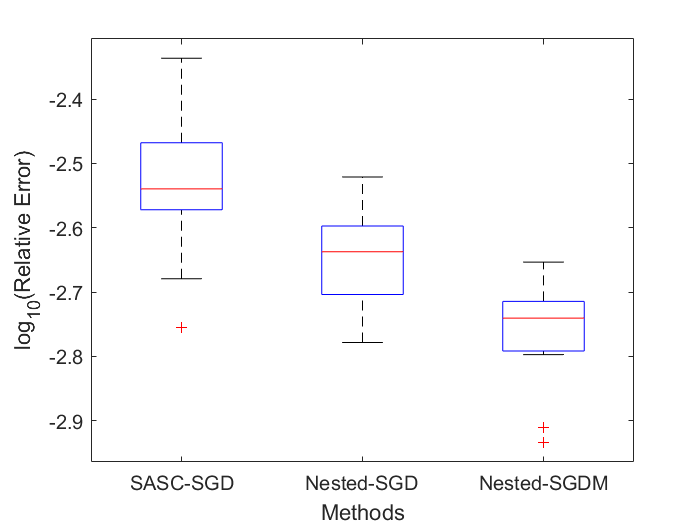}
\caption{Relative errors of 3 methods after 1e7 iterations in 20 simulations.}
\label{Fig_qp_bp_nested}
\end{minipage}%
\begin{minipage}[ht]{0.5\linewidth}
\centering
\includegraphics[width=1\textwidth]{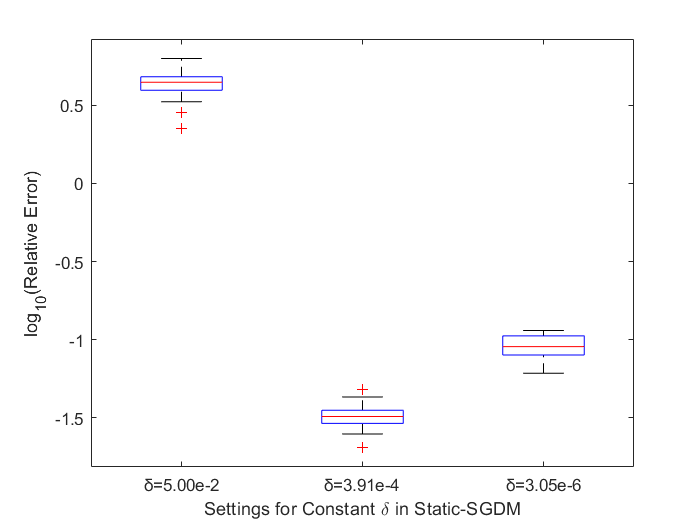}
\caption{Relative errors of \texttt{Static-SGDM} after 1e7 iterations in 20 simulations.}
\label{Fig_qp_bp_static}
\end{minipage}
\end{figure}

\cref{Fig_qp_os} shows the convergence of the relative error for the
three methods with the best parameter settings and \texttt{Static-SGDM} (using the geometric average \(\delta\)), respectively.
In \cref{Fig_qp_os} we see that using the momentum method in the nested algorithm to solve the sub-problems is more efficient compared to SGD. 
\cref{algo_nested_linear} performs slightly better with SGD compared to using SASC to solve the sub-problems. For this run, even though \(\|\lambda_{\text{cvx}}\|_{\infty}=0.027\), we use \(\xi =1\), 
which indicates that, in practice, \(\xi \) does not need to be a close upper bound of \(\|\lambda^*\|_{\infty}\). 
It is also shown in \cref{sec:details of numerical} that the \texttt{Nested-SGD} and \texttt{Nested-SGDM} are more robust than \texttt{SASC-SGD} when the parameters are tuned. Note that \texttt{Static-SGDM} does not converge because the accuracy of the penalty reformulation is directly proportional to \(\delta\). On the other hand, if we set \(\delta\) in \texttt{Static-SGDM} to be the last \(\delta\) in \texttt{Nested-SGDM}, as shown in \cref{Fig_qp_bp_static}, its performance gets worse due to the large condition number, which is inversely proportional to \(\delta\). These results confirm the advantage of the nested algorithm.

The availability of duality gap is yet another advantage of our approach.
\cref{Fig_qp_dg} shows the convergence of the minimal duality gap for the three algorithms.
During iterations, whereas the primal solutions are not necessarily feasible, the dual solutions obtained by \eqref{eq:Dual Variable} are feasible. Hence, we use \(F_{\xi ,0}(x_t)-G(\lambda_t)\) as the duality gap 
(see \cref{sec:duality}) 
and plot \(\min_{1\le s \le t} [F_{\xi ,0}(x_s)-G(\lambda_s)]\) with respect to \(t\) in \cref{Fig_qp_dg}. In this figure it is clear that the proposed nested algorithms perform better than the \texttt{Static-SGDM}.

\begin{figure}[ht]
\begin{minipage}[ht]{0.5\linewidth}
\centering
\includegraphics[width=1\textwidth]{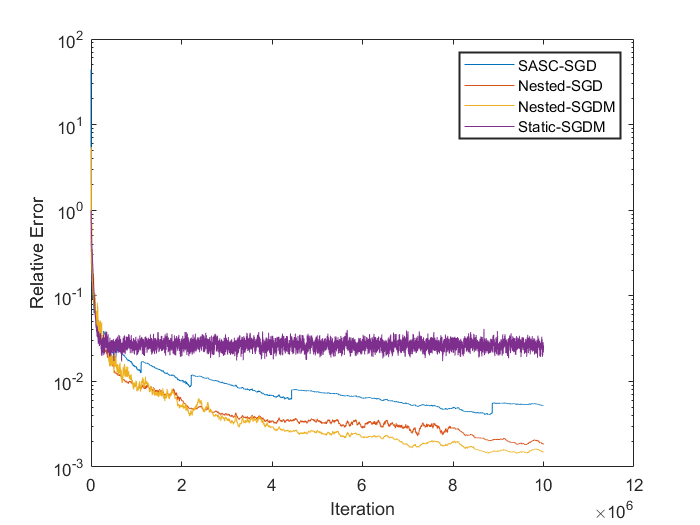}
\caption{Convergence of relative error.}
\label{Fig_qp_os}
\end{minipage}%
\begin{minipage}[ht]{0.5\linewidth}
\centering
\includegraphics[width=1\textwidth]{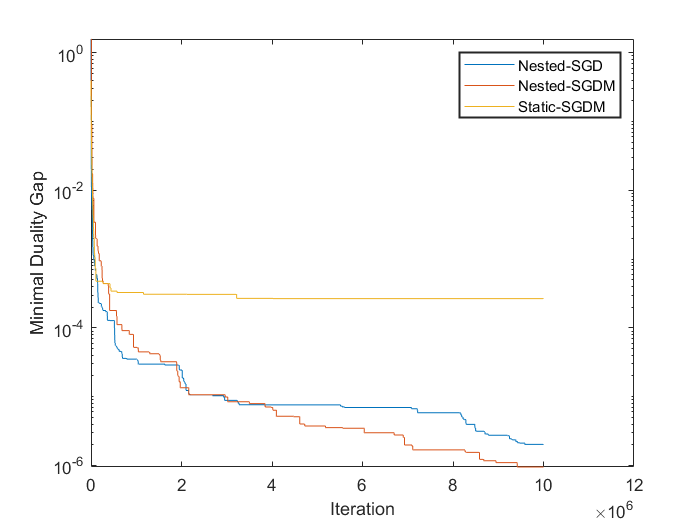}
\caption{Convergence of duality gap.}
\label{Fig_qp_dg}
\end{minipage}
\end{figure}

\subsection{Support Vector Machine}
\label{subsec:svm}
In this subsection, we consider the following hard margin support vector machine (SVM) problem for classification:
\begin{equation}\label{eq:SVM}
\begin{aligned}
\min_{x\in\mathbb{R}^n}\quad&\frac12\|x\|_2^2\\
\mathrm{s.t.}\quad &b_i\left<a_i,x\right>\ge1,\ i=1,...,m,
\end{aligned}
\end{equation}
where \(\{a_1,...,a_m\}\subset\bbR^n\) and \(\{b_1,...,b_m\}\subset\{-1,1\}^n\) are the features and labels of the observations. For the experiment, we use the \texttt{mushrooms} dataset of libsvm database \cite{CC01a}, with 8,124 observations and 112 features (the features are normalized before training). Features in \texttt{mushrooms} are separable, and we use \texttt{CVX} calling \texttt{gurobi} to produce an approximate solution to \eqref{eq:SVM}, \(x_{\text{cvx}}\).

\cref{Fig_svm_rler} shows the convergence of relative errors of the solutions, i.e., \(\frac{\|x-x_{\text{cvx}}\|_2}{\|x_{\text{cvx}}\|_2}\), for the four algorithms considered in \cref{subsec:qp}, whereas \cref{Fig_svm_dg} shows the convergence of the duality gaps. In these figures we observe that, consistent with the results in \cref{subsec:qp}, \texttt{Nested-SGDM} converges the fastest, followed next by \texttt{Nested-SGD}. 

\begin{figure}[ht]
\begin{minipage}[ht]{0.5\linewidth}
\centering
\includegraphics[width=1\textwidth]{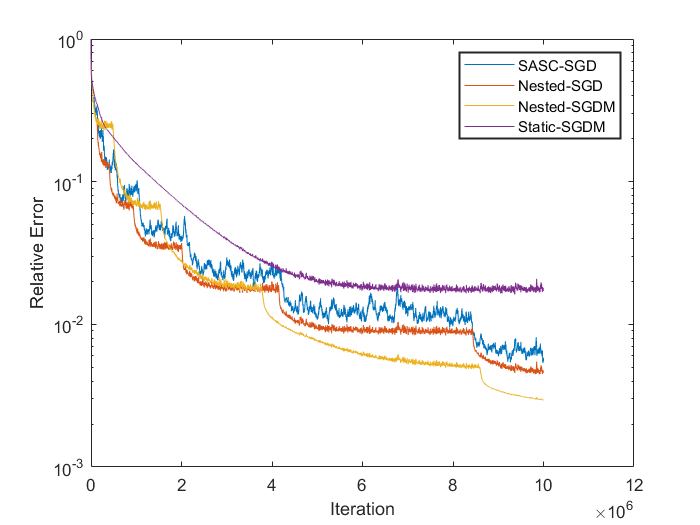}
\caption{Convergence of relative errors.}
\label{Fig_svm_rler}
\end{minipage}%
\begin{minipage}[ht]{0.5\linewidth}
\centering
\includegraphics[width=1\textwidth]{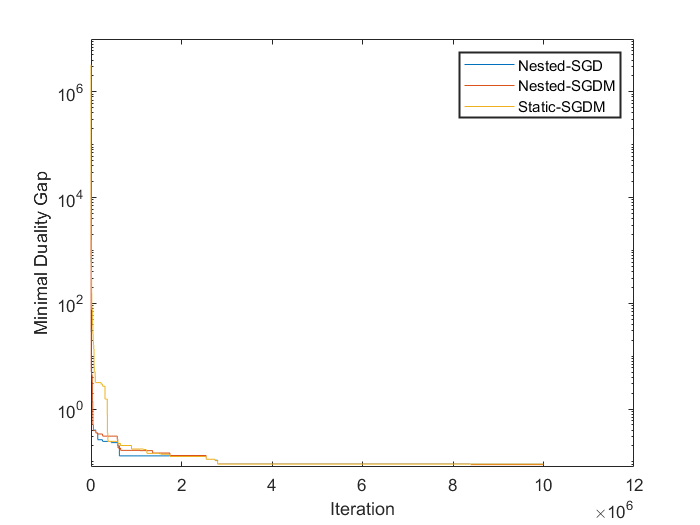}
\caption{Convergence of duality gaps.}
\label{Fig_svm_dg}
\end{minipage}
\end{figure}

To test the safety and effectiveness of the screening procedure, in \cref{Fig_svm_rler_screening}, we compare algorithms \texttt{SASC-SGD}, \texttt{Nested-SGD}, \texttt{Nested-SVRG}, \texttt{Nested-SVRG-Screening}. 
In this figure, we observe that even though without screening \texttt{Nested-SVRG} performs worse than \texttt{Nested-SGD}, with screening it outperforms \texttt{Nested-SGD}. The reason for the performance improvement is that we dynamically drop a significant amount of redundant constraints in \cref{algo_nested_linear_screening}, which saves iterations that query these constraints. In other words, the screening procedure in \cref{algo_nested_linear_screening} is effective in detecting and eliminating redundant constraints. Details on the constraints that are dropped and kept in \texttt{Nested-SVRG-Screening} are presented in \cref{sec:details of numerical}. 

\begin{figure}[ht]
\centering
\includegraphics[width=0.5\textwidth]{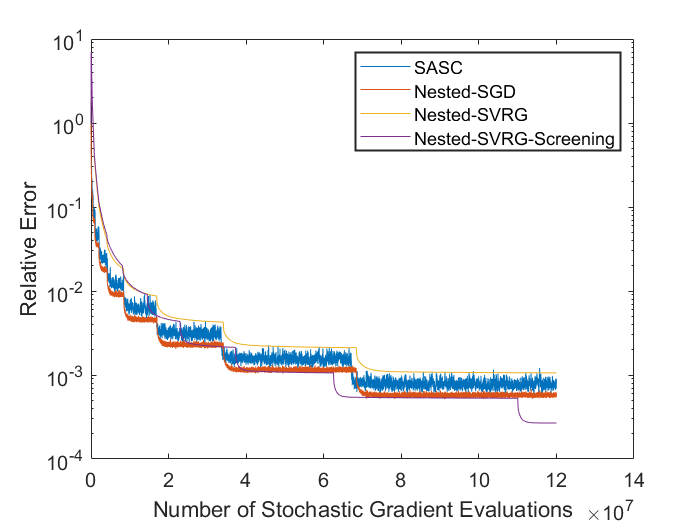}
\caption{Relative errors during iterations}
\label{Fig_svm_rler_screening}
\end{figure}

\section{Conclusions}
\label{sec:conclusions}
We give a nested accelerated stochastic method for penalized reformulations of strongly convex function minimization subject to linear constraints. 
The algorithm generates a primal solution within a distance of $\epsilon$ to the optimum and an \(O(\epsilon)\) sub-optimal dual solution in an expected $\tilde O(1/\sqrt{\epsilon})$ stochastic first-order iterations. We also design the screening procedure for our algorithm, which can effectively drop inactive constraints. Computational results on synthetic quadratic programming instances and support vector machine instances demonstrate the effectiveness and robustness of the nested algorithm.

\appendix

\section{Lemmas and Missing Proofs}
\label{sec:proofs}
\subsection{Proof of \cref{prop:softplus}}
\begin{proof}
Observe that
\[\begin{aligned}
\lim_{\delta\to0^+} \delta \log(1+\exp(t/\delta))&=t,\quad t\in(0,\infty),\\
\lim_{\delta\to0^+} \delta \log(1+\exp(t/\delta))&=0,\quad t\in(-\infty,0].
\end{aligned}\]
Hence, we may define \(p_0(t) := \max(0,t)\). Then, \(p_{\delta}(t)\) is continuous with respect to \(\delta\), for \(\delta\in[0,\infty)\).
For \(\delta>0\), we have
\[\begin{aligned}
p_{\delta}'(t)&=\frac{\exp(t/\delta)}{1+\exp(t/\delta)}\in[0,1],\\
p_{\delta}''(t)&=\frac{1}{\delta}\frac{\exp(t/\delta)}{(1+\exp(t/\delta))^2}\in[0,\frac1{4\delta}].\\
\end{aligned}
\]
For \(0<\delta_1\le\delta_2\), \(p_{\delta_1}(t)\le p_{\delta_2}(t)\)
\[p_{\delta_2}'(t)-p_{\delta_1}'(t)\left\{\begin{aligned}
&\ge 0,\quad &t<0,\\
&= 0,\quad &t=0, \\
&\le 0,\quad &t>0.\\
\end{aligned}\right.\]
Then, \(0\le p_{\delta_2}(t)- p_{\delta_1}(t)\le p_{\delta_2}(0)- p_{\delta_1}(0) = (\delta_2-\delta_1)\log2\). Because \(p_{\delta}(t)\) is continuous with respect to \(\delta\) for \(\delta\in[0,\infty)\), letting \(\delta_1  = 0\) the result holds. 
\end{proof}

\subsection{Proof of \cref{lemma:1d 1+exp}}
\begin{proof}
Let \(\tilde \theta_{c,d}=\arg\max_{\theta}\phi_{c,d}(\theta)\). Since 
\[\begin{aligned}
\frac{d \phi_{c,d}(\theta)}{d \theta}=\frac1{(d+c\exp(\theta))^2}((d+c\exp(\theta))-c\theta\exp(\theta)),
\end{aligned}\]
    \(\tilde \theta_{c,d}\) satisfies
    \[\tilde \theta_{c,d}=1+\frac{d}{c}\exp(-\tilde \theta_{c,d}).\]
    Moreover, \(\frac{d \phi_{c,d}(\theta)}{d \theta}>0\) for \(\theta<\tilde \theta_{c,d}\) and \(\frac{d \phi_{c,d}(\theta,\delta)}{d \theta}<0\) for \(\theta>\tilde \theta_{c,d}\).
    When \(\log(\frac{d}{c})\ge 2\),
    \[\frac{d \phi_{c,d}(\theta)}{d \theta}|_{\theta=\log(\frac{d}{c})}\le0.\]
    Then, 
    \[\phi_{c,d}(\theta)\le\phi_{c,d}(\tilde \theta_{c,d})=(\tilde \theta_{c,d}-1)/d\le(\log(\frac{d}{c})-1)/d.\]
\end{proof}

\subsection{Proof of \cref{lemma:stopping criteria:general proximal}}
\begin{proof}
\noindent
1. By the definition, \(\tilde x = \arg\min_{y\in\mathbb{R}^n} \left[\psi(y) + \nabla {\tilde F}({\bar x})^T(y-{\bar x})+\frac{1}{2\alpha}\|y-{\bar x}\|_2^2\right].\) Hence, \[{g_{\psi}(\bar x;\tilde F,\alpha)}-\nabla \tilde F(\bar x)\in\partial\psi(\tilde x),\]
which means \[{g_{\psi}(\bar x;\tilde F,\alpha)}-\nabla \tilde F(\bar x)+\nabla \tilde F(\tilde x)\in\partial F(\tilde x).\]
For any given \(x\), by the \(\mu\)-strong convexity, 
\begin{equation}\label{eq:stopping criteria:general proximal eq3}
\begin{aligned}
F(x)\ge F(\tilde x) + ({g_{\psi}(\bar x;\tilde F,\alpha)}-\nabla \tilde F(\bar x)+\nabla \tilde F(\tilde x))^T(x-\tilde x)+\frac{\mu}{2}\|x-\tilde x\|_2^2.
\end{aligned}    
\end{equation}
Letting \(x=\bar x\) in \eqref{eq:stopping criteria:general proximal eq3}, we obtain
\[\begin{aligned}
F(\bar x)&\ge\ F(\tilde x) + ({g_{\psi}(\bar x;\tilde F,\alpha)}-\nabla \tilde F(\bar x)+\nabla \tilde F(\tilde x))^T(\bar x-\tilde x)+\frac{\mu}{2}\|\bar x-\tilde x\|_2^2\\
&=\ F(\tilde x) + (\frac{\mu\alpha^2}{2}+\alpha)\|{g_{\psi}(\bar x;\tilde F,\alpha)}\|_2^2-(\nabla \tilde F(\bar x)-\nabla \tilde F(\tilde x))^T(\bar x-\tilde x)\\
&\ge\ F(\tilde x) + (\alpha-L\alpha^2)\|{g_{\psi}(\bar x;\tilde F,\alpha)}\|_2^2\\
&\ge\ F(\tilde x) + \frac{\alpha}{2} \|{g_{\psi}(\bar x;\tilde F,\alpha)}\|_2^2.
\end{aligned}\]
Therefore, \(\|{g_{\psi}(\bar x;\tilde F,\alpha)}\|_2^2\le \frac{2}{{\alpha}}({F}(\bar x)-{F}(\tilde x))\le \frac{2}{{\alpha}}({F}(\bar x)-{F}^*).\)

\noindent
2. Minimizing the right-hand side of \eqref{eq:stopping criteria:general proximal eq3}, we find that \(\forall y\in\mathbb{R}^n\),
\[\begin{aligned}
{F}(y)\ge &\ \min_x\left[ ({g_{\psi}(\bar x;\tilde F,\alpha)}-\nabla \tilde F(\bar x)+\nabla \tilde F(\tilde x))^T(x-\tilde x)+\frac{\mu}{2}\|x-\tilde x\|_2^2\right]\\
\ge &\ {F}(\tilde x)-\frac{1}{2\mu}\|{g_{\psi}(\bar x;\tilde F,\alpha)}-\nabla \tilde F(\bar x)+\nabla \tilde F(\tilde x)\|_2^2\\
= &\ {F}(\tilde x)-\frac{1}{2\mu}\|{g_{\psi}(\bar x;\tilde F,\alpha)}\|_2^2 + \frac1\mu {g_{\psi}(\bar x;\tilde F,\alpha)}^T(\nabla \tilde F(\bar x)-\nabla \tilde F(\tilde x))\\&\ -\frac1{2\mu}\|\nabla \tilde F(\bar x)+\nabla \tilde F(\tilde x)\|_2^2\\
\ge &\ {F}(\tilde x)-\frac{1}{2\mu}\|{g_{\psi}(\bar x;\tilde F,\alpha)}\|_2^2 + \left(\frac{1}{\mu\alpha L}-\frac1{2\mu}\right)\|\nabla \tilde F(\bar x)+\nabla \tilde F(\tilde x)\|_2^2,\\ 
\end{aligned}\]
where the last inequality follows from the co-coercivity of \(\nabla \tilde F\). Hence, \(\|{g_{\psi}(\bar x;\tilde F,\alpha)}\|_2^2 \ge 2\mu ({F}(\tilde x)-{F}^*)\).
\end{proof}

\subsection{Complexity of Proximal SVRG with Catalyst Acceleration}
\ignore{
To optimize an unconstrained strongly convex problem, for the accelerated proximal gradient method, particularly, the scheme \(\mathcal{A}(x_0,L_0,\mu_{\psi})\) in \cite{nesterov2013gradient}), the following lemma holds.

\begin{lemma}\label{lemma:Accelerated Full Gradient Lemma}
Consider applying the accelerated proximal gradient method in \cite{nesterov2013gradient} on the problem \(\min_x\Phi(x)=\phi(x)+\psi(x)\), where \(\Phi(x)\) is \(\mu\)-strongly convex, \(\phi(x)\) is convex and $L$-smooth, and $\psi(x)$ is a convex and proximable function. Then, the number of iterations till \(\Phi(x_t)-\Phi^*\le \epsilon\) is bounded by
\[O\left(\sqrt{\frac{L+\mu}{\mu}}\log((L+\mu)\frac{\|x_0-x^*\|}\epsilon)\right).\]
\end{lemma}
}
In this section, we review the main complexity result of the proximal SVRG method with catalyst acceleration \cite{lin2015universal} that is used in the analysis of \cref{algo_nested_linear}, particularly in the proof of \cref{prop:Nested Complexity-SVRG Catalyst}. Consider the problem 
\begin{equation}\label{eq:poi_svrg}
\min_{x \in \bbR^n}\quad \Phi(x):=\frac1m\sum_{i=1}^m \phi_i(x)+\psi(x),
\end{equation}
where \(\phi_i : \bbR^n \to \bbR\) are convex $L_i$-smooth functions, and $\psi : \bbR^n \to \bbR$ is a convex proximal function. Namely, it is assumed that the proximal operator of $\psi$, defined by
\[\mathrm{prox}_{t\psi}(x) := \arg\min_{u}\psi(u)+\frac1{2t}\|x-u\|_2^2,\]
is computable for any $x \in \bbR^n$ and $t > 0$. 
Lin et al. \cite{lin2015universal} propose the catalyst acceleration technique and develop the following complexity result, in the strongly convex case, for the proximal SVRG method with catalyst acceleration. Note that this complexity result is stated in terms of the number of evaluations of the the gradients of the individual component functions $\phi_i$ and of the proximal operator $\mathrm{prox}_{t\psi}$. In the result below, the \(\tilde O\) notation hides universal constants and logarithmic dependencies in \(\mu\), \(L\) and \(m\).

\begin{lemma}[Lemma C.1 of \cite{lin2015universal}]\label{lemma:Catalyst SVRG Lemma}
Consider applying the proximal SVRG method with catalyst acceleration to problem \eqref{eq:poi_svrg}. Suppose that \(\Phi\) is \(\mu\)-strongly convex for some $\mu > 0$, and let the average Lipschitz constant of the gradients $\nabla \phi_i$ be defined by $L := \frac{1}{m}\sum_{i = 1}^m L_i$. 
Then, the expectation of the number of evaluations of gradients $\nabla \phi_i$ and proximal mappings required to satisfy
\(\Phi(x_t)-\Phi^*\le \epsilon\) is upeer bound bounded by
\[\tilde O\left(\left(m+\sqrt{\frac{mL}{\mu}}\right)\log\left(\frac{\Phi(x_0)-\Phi^*}{\epsilon}\right)\right),\]
where \(\tilde O\) hides universal constants and logarithmic dependencies in \(\mu\), \(L\) and \(m\).
\end{lemma}

\section{Details of Experiments}
\label{sec:details of numerical}
\subsection{Momentum Method}
For \texttt{Nested-SGDM}, we use the momentum method in \cite{goodfellow2016deep}. The algorithm to solve the problem \(\min_{\zeta}\frac1m\sum_{i=1}^m\phi_i({}{x})\), 
is stated as
\begin{equation}\label{Nesterov Momentum}
\begin{aligned}
{}{v} &\leftarrow \alpha {}{v}-t \nabla_{x}\phi_i({}{x}+\alpha{}{v}), \\
{}{x} &\leftarrow {}{x}+{}{v},
\end{aligned}
\end{equation}
where \(t\) is the step size, and \(\alpha\) is the ratio related to the momentum, which is often set to be \((0.5,0.8,0.9,0.95)\). We use \eqref{Nesterov Momentum} as a simpler representative for the catalyst accelerated SVRG method.

\subsection{Details for \cref{subsec:qp}}
\label{subsec:details qp}
For \texttt{SASC-SGD} and \texttt{Nested-SGD}, the parameters include: \(\eta\) (or \(\omega\)), the ratios between 2 consecutive inner loops (see Algorithm 1 in \cite{fercoq2019almost} and \cref{algo_nested_linear}), from \{2,4,8\};  multipliers for the theoretical inner-loop iteration limit used in the code, from \{1,0.6,0.36\}. The theoretical inner-loop iteration limit for \texttt{Nested-SGD} is \(\log(2\eta-1)(\frac{L}{\mu}+\frac{{m}\xi }{4\mu\delta})\), and that for \texttt{Nested-SGDM} is \((2\log(2\eta-1)+\log(\frac{L}{\mu}+\frac{{m}\xi }{4\mu\delta}))\sqrt{{m}(\frac{L}{\mu}+\frac{{m}\xi }{4\mu\delta})}\). Similarly, the step sizes for the algorithms are \(\frac{1}{L+\mu+\frac{{m}\xi }{4\delta}}\). For SASC, see the settings in Algorithm 1 in \cite{fercoq2019almost}. For \texttt{Nested-SGDM}, parameter \(\alpha\) in \eqref{Nesterov Momentum} is taken as \{0.5,0.8,0.9\}. The following figures show the relative errors of the three methods with different parameters after 1e7 iterations (starting from 0 point) for \eqref{eq:New QP}, in 20 simulations (including the generation of the problems and the stochastic sequence of indexes). For \texttt{Nested-SGD} and \texttt{Nested-SGDM}, \(\xi =1\) and \(\delta_0=0.05\). We pick (2,0.36), (4,0.6), and (2,1,0.9) as the best parameter settings for \texttt{SASC-SGD}, \texttt{Nested-SGD}, and \texttt{Nested-SGDM}, considering the medium and largest errors (see \cref{fig:app:1:side:a}--\cref{fig:app:3:side:a}). We also see here that \texttt{SASC-SGD} is more likely to diverge for different parameter settings. The setting for the initial step size for \texttt{SASC-SGD} is \(\frac{1}{10L}\), instead of \(\frac3{4L}\) in Algorithm 1 in \cite{fercoq2019almost}, which results in divergence. For \texttt{Static-SGDM}, \(\xi =1\) and the choices of \(\delta\) have been shown. The duality gaps in \cref{Fig_qp_dg} are calculated every 1000 steps.

\begin{figure}[ht]
\begin{minipage}[ht]{0.5\linewidth}
\centering
\includegraphics[width=1\textwidth]{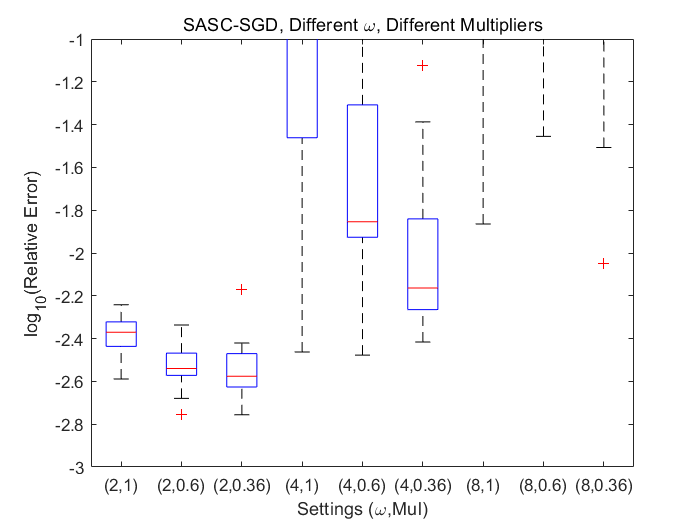}
\caption{Relative errors of \texttt{SASC-SGD}}
\label{fig:app:1:side:a}
\end{minipage}%
\begin{minipage}[ht]{0.5\linewidth}
\centering
\includegraphics[width=1\textwidth]{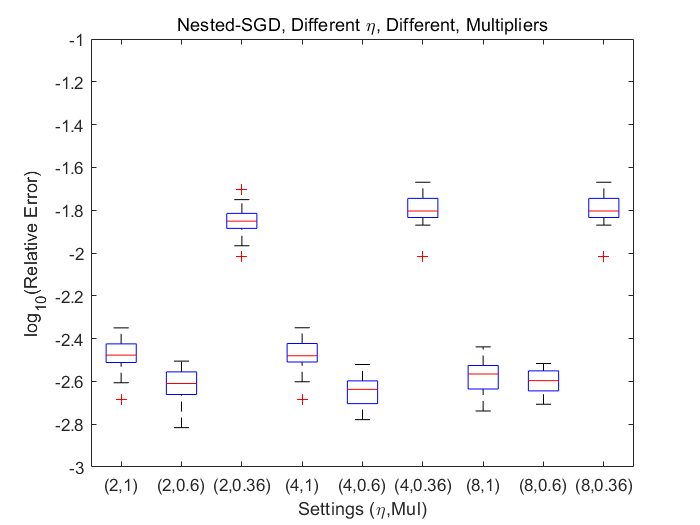}
\caption{Relative errors of \texttt{Nested-SGD}}
\label{fig:app:1:side:b}
\end{minipage}
\end{figure}

\begin{figure}[ht]
\begin{minipage}[ht]{0.5\linewidth}
\centering
\includegraphics[width=1\textwidth]{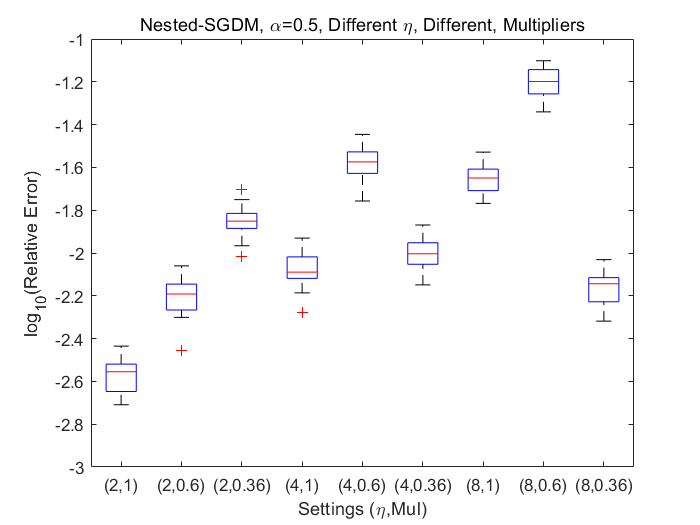}
\caption{Relative errors of \texttt{SASC,SGDM} (\(\alpha=0.5\))}
\label{fig:app:2:side:a}
\end{minipage}%
\begin{minipage}[ht]{0.5\linewidth}
\centering
\includegraphics[width=1\textwidth]{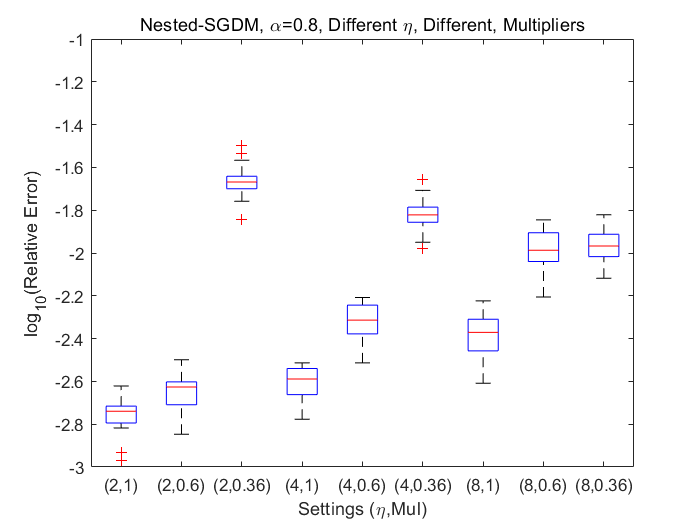}
\caption{Relative errors of \texttt{SASC,SGDM} (\(\alpha=0.8\))}
\label{fig:app:2:side:b}
\end{minipage}
\end{figure}

\ignore{
\begin{figure}[ht]
\begin{minipage}[ht]{0.5\linewidth}
\centering
\includegraphics[width=1\textwidth]{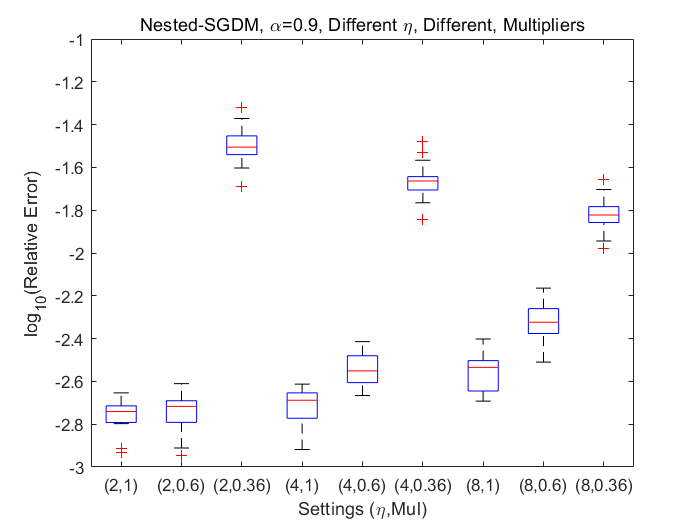}
\caption{Relative errors of \texttt{SASC,SGDM} (\(\alpha=0.9\)) with different parameters}
\label{fig:app:3:side:a}
\end{minipage}%
\begin{minipage}[ht]{0.5\linewidth}
\centering
\includegraphics[width=1\textwidth]{SVM_gap_100401.png}
\caption{Duality gaps during iterations (SVM, \cref{Fig_svm_dg} with larger y-axis scale)}
\label{fig:app:3:side:b}
\end{minipage}
\end{figure}
}

\begin{figure}[ht]
\begin{minipage}[ht]{0.5\linewidth}
\centering
\includegraphics[width=1\textwidth]{Fig_Exp1_Appendix_Nested_SGDM_3_1004.png}
\caption{Relative errors of \texttt{SASC,SGDM} (\(\alpha=0.9\))}
\label{fig:app:3:side:a}
\end{minipage}%
\begin{minipage}[ht]{0.5\linewidth}
\centering
\includegraphics[width=1\textwidth]{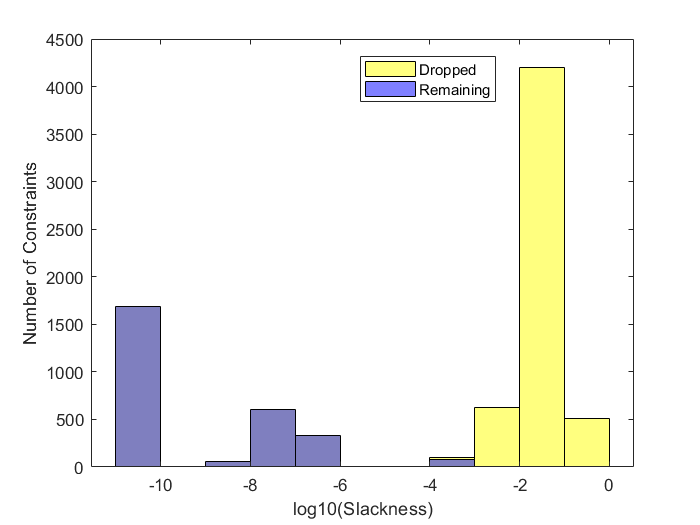}
\caption{Slackness of dropped and remaining constraints}
\label{Fig_svm_screening_tightness}
\end{minipage}
\end{figure}

\ignore{
\begin{figure}[ht]
\centering
\includegraphics[width=0.5\textwidth]{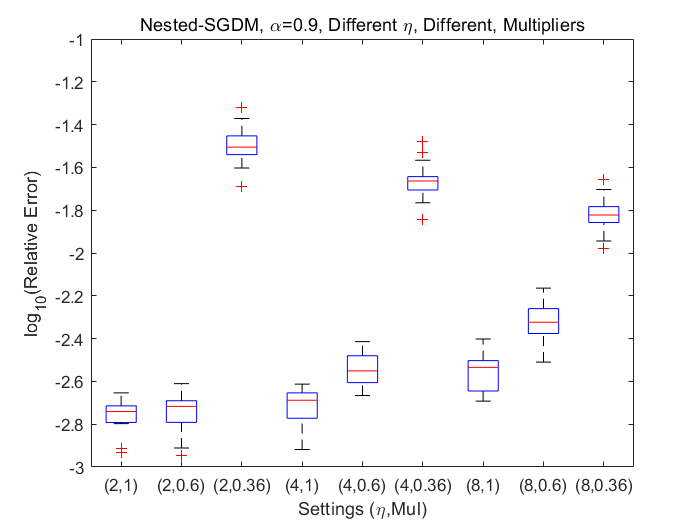}
\caption{Relative errors of \texttt{SASC,SGDM} (\(\alpha=0.9\)) with different parameters}
\label{fig:app:3:side:a}
\end{figure}}

\subsection{Details for \cref{subsec:svm}}
\label{subsec:details svm}
The starting point for all methods is \(x_0=0\). For \texttt{SASC-SGD}, the parameter setting is the same with 6.3 in \cite{fercoq2019almost}. For \texttt{Nested-SVRG} and \texttt{Nested-SVRG-Screening}, the full gradients are calculated every \(5m\) stochastic gradient steps, where \(m\) is the numbers of constraints or remaining constraints. The numbers of iterations for these two algorithms are approximately 5/6 of other algorithms to maintain the same number of stochastic gradient evaluations. For \texttt{Nested-SGD}, \texttt{Nested-SVRG}, \texttt{Nested-SVRG-Screening}, and \texttt{Nested-SGDM}, \(\xi =1\) and \(\delta_0=0.005\). And other parameter settings are (2,0.3), (2,0.5), (2,0.5), and (4,0.5,0.5) for \texttt{Nested-SGD}, \texttt{Nested-SVRG}, \texttt{Nested-SVRG-Screening}, and \texttt{Nested-SGDM}. Here, the theoretical inner-loop iteration numbers for \texttt{Nested-SGD}, \texttt{Nested-SVRG},\\ \texttt{Nested-SVRG-Screening} are \(\log(2\eta-1)(\frac{L}{\mu}+\frac{{m}\xi }{4\mu\delta})\), and that for \texttt{Nested-SGDM} is \((2\log(2\eta-1)+\log(\frac{L}{\mu}+\frac{{m}\xi }{4\mu\delta}))\sqrt{{m}(\frac{L}{\mu}+\frac{{m}\xi }{4\mu\delta})}\). Similarly, the step sizes for the \texttt{Nested-SGD}, \texttt{Nested-SVRG}, \texttt{Nested-SVRG-Screening} are \(\frac{4}{L+\mu+\frac{{m}\xi }{4\delta}}\), and for \texttt{Nested-SGDM} they are \(\frac{2}{L+\mu+\frac{{m}\xi }{4\delta}}\) (considering the momentum term) respectively. (We checked increasing the step sizes and (or) decreasing inner-loop iteration numbers for \texttt{SASC-SGD}, which turned out not to bring significantly better performance.) The duality gaps in \cref{Fig_svm_dg} are calculated every 1000 steps. 

In the experiment of \cref{Fig_svm_rler_screening}, for simplicity, constraint dropping criterion is set as \(a_i^T\tilde x_t-b_i < -2\sqrt{m_t}\delta_t\). As such, 2,777 out of 8,198 constraints are kept in the last inner loop of \texttt{Nested-SVRG-Screening}. In \cref{Fig_svm_screening_tightness}, the slackness of the constraints, (\(b_i-a_i^Tx_{\text{cvx}}\)), \(i=1,...,m\) are compared (to include the constraints that \(b_i-a_i^Tx_{\text{cvx}}<0\), the slackness values are set as 1e-11). Constraints with slackness greater than 1e-3 (5,536 out of 8,198) are dropped and all the constraints with slackness less than 1e-7 (2,358 out of 8,198) are kept. 

\ignore{
\begin{figure}[ht]
\centering
\includegraphics[width=0.5\textwidth]{Fig_Tightness_1204.png}
\caption{Comparison of tightness of original constraints and remaining constraints}
\label{Fig_svm_screening_tightness}
\end{figure}
}

\section*{Acknowledgments}
This research is supported, in part, by NSF AI Institute for Advances
in Optimization Award 211253, DOD ONR grant 12951270, and NSF Awards CCF-1755705 and CMMI-1762744.

\bibliographystyle{siamplain}
\bibliography{Newbib}
\end{document}